\documentclass[a4paper]{scrartcl}
\usepackage{amsthm}
\usepackage{amsmath}
\usepackage{amssymb}
\usepackage[colorlinks,linkcolor={red!50!black},citecolor={blue!50!black},urlcolor={blue!80!black}]{hyperref}

\usepackage{enumerate}
\usepackage{authblk}
\usepackage[numbers,square,sort&compress,longnamesfirst]{natbib}
\usepackage{booktabs}
\title{A High-Performance Parallel Algorithm for Multi-Objective Integer Optimization}

\author{Kathrin Prinz$^*$}
\author{Levin Nemesch}
\author{Stefan Ruzika}
\affil{RPTU Kaiserslautern-Landau, Germany\\
$^*$Corresponding author, prinz.kathrin@math.rptu.de}

\usepackage{mathtools,amsmath,amssymb,amsfonts}
\usepackage{hyperref}
\usepackage[capitalize,noabbrev]{cleveref}
\usepackage[caption=false]{subfig}
\usepackage{tikz}
\usetikzlibrary{positioning,shapes.misc,patterns}
\usepackage{pgfplots}

\usepackage{lipsum}
\usepackage{amsfonts}
\usepackage{graphicx}
\usepackage{epstopdf}
\usepackage{algorithmic}
\usepackage{enumitem}
\usepackage[ruled,linesnumbered]{algorithm2e}
\usepackage{nicefrac}

\Crefname{algocf}{Algorithm}{Algorithms}

\theoremstyle{plain}
\newtheorem{theorem}{Theorem}
\newtheorem{remark}[theorem]{Remark}
\newtheorem{lemma}[theorem]{Lemma}
\newtheorem{proposition}[theorem]{Proposition}
\newtheorem{corollary}[theorem]{Corollary}

\theoremstyle{definition}

\newtheorem{assumption}{Assumption}
\newtheorem{definition}[theorem]{Definition}

\newtheorem{example}[theorem]{Example}

\newcommand{\e}{\varepsilon}
\newcommand{\llex}{<_{\text{lex}}}
\newcommand{\emethod}[1]{\Pi(#1)}
\newcommand{\E}[1]{E(#1)}
\newcommand{\clE}[2]{E^{#2}(#1)}
\newcommand{\vc}[1]{(#1^1,\dots,#1^{k-1})}
\newcommand{\MOIP}{\textbf{MOIP}}
\newcommand{\cl}[1]{\mathrm{cl}(#1)}
\newcommand{\Y}{\mathcal{Y}}
\newcommand{\Z}{\mathcal{Z}}
\newcommand{\V}[1]{\mathcal{V}(#1)}
\newcommand{\U}{\mathcal{U}}
\newcommand{\W}{\mathcal{W}}
\newcommand{\Li}{\mathcal{L}}
\newcommand{\T}{\mathcal{T}}
\newcommand{\scion}[2]{\mathrm{scion}^{#2}(#1)}

\definecolor{i1}{RGB}{227,27,76}
\definecolor{i5}{RGB}{6,66,132}
\definecolor{i2}{RGB}{255,162,82}
\definecolor{i4}{RGB}{119,182,186}
\definecolor{i6}{RGB}{76,53,117}
\definecolor{i3}{RGB}{106,178,231}
\definecolor{tree}{RGB}{38,208,124}

\definecolor{dummy}{RGB}{50,50,50}
\DeclareMathOperator*{\lexmin}{\text{lex}\,min}
\DeclareMathOperator*{\arglexmin}{arg\,lex\,min}
\DeclareMathOperator*{\argmin}{arg\,min}
\DeclareMathOperator*{\argmax}{arg\,max}

\newcommand\cuboid[8]{
	\filldraw[draw=#7,fill opacity=#8,fill=#7!20,very thick] (axis cs: #1,#2,#3) -- (axis cs:#1,#2,#6) -- (axis cs:#1,#5,#6) -- (axis cs:#1,#5,#3) -- cycle;
	\filldraw[draw=#7,fill opacity=#8,fill=#7!20,very thick] (axis cs: #1,#2,#3) -- (axis cs: #1,#5,#3) -- (axis cs: #4,#5,#3) -- (axis cs: #4,#2,#3) -- cycle;
	\filldraw[draw=#7,fill opacity=#8,fill=#7!20,very thick] (axis cs: #1,#2,#3) -- (axis cs: #1,#2,#6) -- (axis cs: #4,#2,#6) -- (axis cs: #4,#2,#3) -- cycle;
}
\begin{document}
\maketitle
	\begin{abstract}
    Multi-objective integer optimization problems are hard to solve, mainly because the number of nondominated images is often extremely large. We present the first exact algorithm, called PEA, that fully utilizes the multicore architecture of modern hardware. By exploiting the structure of the parameter set of the underlying scalarization, PEA can use a high number of threads while avoiding the usual pitfalls of parallel computing. It is highly scalable and easy to implement. As a result, PEA can solve much larger instances than previous state-of-the-art algorithms. Besides, PEA has a sound theoretical foundation. Unlike other existing parallel algorithms, it always solves the same number of scalarization problems as comparable sequential algorithms. We demonstrate the potential of PEA in a computational study.
    

	\end{abstract}

\section{Introduction}
Many industrial problems involve multiple, often conflicting objectives. 
Such problems can be modeled as \emph{multi-objective optimization problems}, thus, optimizing all objectives simultaneously.
As a result, there is no single optimal solution, but a whole set of so-called \emph{Pareto optimal} or \emph{efficient} solutions. The corresponding images under the objective functions are called \emph{nondominated}.
Consequently, ''solving`` a multi-objective optimization problem requires computing the entire \emph{nondominated set}~$Y_N$ and imposes a huge computational burden: Most such problems are known to be \emph{intractable}, i.\,e., the number of nondominated images is exponential in the  instances size.

State-of-the-art exact algorithms for multi-objective integer optimization problems work in the \emph{objective space} and iteratively transform the problem into a series of single-objective problems via \emph{scalarizations}, cf.~\cite{Boland2017a,Ozlen2014,Daechert2024}. 
This makes it possible to  use highly efficient single-objective problem solvers such as Gurobi~\cite{GurobiOptimization2023} and CPLEX~\cite{IBM2023}. 
With very few exceptions, such algorithms are sequential.

However, modern processors are not designed for purely sequential computing anymore: 
It is hard to find a processor, even in handheld devices, that does not incorporate multiple cores that can execute operations in parallel.
So far, multi-objective optimization barely takes advantage of this potential. 
Meanwhile, the common practice of solving many single-objective scalarization problems already provides a good basis for parallelization.
Why not enumerate the scalarization problems in a way that allows them to be solved independently and in parallel?
Unfortunately, conventional algorithms such as in~\cite{Daechert2024,Ozlen2014} are not well-suited for this: They require complex data structures to avoid solving unnecessary scalarization problems. Thus, the scalarization problems cannot be solved independently and solving them in parallel requires costly information sharing between threads. 

We propose a shift in paradigm: Algorithm designs for multi-objective optimization should be inherently suited for parallel computing. We demonstrate that it is possible to design such a  parallel algorithm for integer problems --- without the need for complex data structures and with independently running scalarization problems. Threads can operate largely autonomously and without much communication. 
Hereby, a guided search in the parameter space of a suitable scalarization rather than in the objective space achieves superior performance.

\subsection{Literature}
We briefly review relevant literature on exact algorithms for multi-objective integer problems.
For a more extensive general overview of exact algorithms for multi-objective (mixed-)integer problems, we refer to the recent survey by Halffmann et al.~\cite{Halffmann2022}.

Many methods to compute the entire nondominated set for multi-objective integer optimization problems work either by restricting the search  to the part of the image space that may contain undiscovered nondominated images, the so-called \emph{search region}, or by searching through the parameter space defined by a scalarization. 

One possible approach to do this is to remove dominated areas in the image space by introducing additional constraints and variables for each discovered nondominated image. Thus, a method relying on such a description needs to iteratively solve increasingly complex single-objective problems~\cite{KLEIN1982,Sylva2004,Lokman2013}.

Another possibility is to describe the search region by so-called local upper bounds.
Here, a set of bounds limits the search region.
This set is iteratively refined with each new nondominated image that is discovered (cf.~\cite{Daechert2015,Klamroth2015,Daechert2017}). 
Upper bound based methods can be combined with different scalarization, see Tamby and Vanderpooten~\cite{Tamby2020} and D{\"a}chert et al.~\cite{,Daechert2024}.


Methods working in the parameter space include the Quadrant Shrinking Method by Boland et al.~\cite{Boland2017}, the L-Shape Method by Boland et al.~\cite{Boland2016}, the algorithm by Kirlik and Sayın~\cite{Kirlik2014} and the recursive algorithm by Ozlen et al.~\cite{Ozlen2014}.

However, between all these approaches, methods that are specifically designed for parallelization are rare:
Dhaenens et al.~\cite{Daechert2015}, extend the algorithm for bi-objective problems by Lemesre et al.~\cite{Lemesre2007} to any number of objectives. First, the search space is divided into equal areas according to one objective, which can be searched in parallel. Then, the rectangles between neighboring images are explored.
Turgut et al.~\cite{Turgut2019} present an exact parallel objective space decomposition algorithm that exploits regional dominance relations between decomposed partitions for pruning.
 The most competitive  exact parallel method so far has been developed by Pettersson et al.~\cite{Pettersson2020}. They introduce a permutation parallelization technique, whereby each thread starts generating the nondominated set using different permutations of the objective functions and shares generated bounds with the other threads.
The practical viability of their method is demonstrated in a computational study.
The most recent algorithm is the tri-objective Parallel Enumeration Algorithm by Ruzika and Prinz~\cite{Prinz2024}, which show superior performance but is limited to only three objectives.

\subsection{Our Contribution}

We extend the Parallel Enumeration Algorithm (PEA) for tri-objective optimization problems by Prinz and Ruzika~\cite{Prinz2024} to any number of objectives\footnote{Note that parts of this work have also appeared in the PhD thesis by Prinz~\cite{Prinz2026}}.
The resulting algorithm is able to distribute scalarization problems among tasks working in parallel.
Each task is able to independently generate their own follow-up scalarization problems, and different tasks rarely need to communicate.
Our algorithm is the first of its kind for any number of objectives and marks an important milestone for the development of new parallel algorithms in multi-objective optimization. 
Unlike existing parallel algorithms, now matter how many threads are used, the number of scalarization problems solved by PEA is always the same, that is, PEA never does any additional or unnecessary work. PEA does not require any complex data structures, and its implementation is straightforward. In addition, PEA is the first parallel algorithm that solves the same number of scalarization problems as state-of-the-art sequential algorithms, which is in $\mathcal{O}(|Y_N|^{\lfloor\frac{k}{2}\rfloor})$ for problems with $k$ objective functions, cf.~\cite{Klamroth2015}.

The approach of PEA can be seen as follows:
The parameters  which are needed for a lexicographic variant of the epsilon-constraint scalarizations throughout the algorithm can be arranged as a directed tree.
Each parameter represents a vertex in the tree, and the arcs between them represent a partial order.
Traversing such a tree leads to an intuitive parallelization, since the knowledge of the root of a subtree is sufficient to start the traversal of this subtree independently from the rest of the larger tree.
This approach enables PEA to solve multi-objective integer optimization problems of unprecedented size in a short amount of time, which is demonstrated in a computational study.  

This paper is organized as follows:
In \Cref{sec:pre}, we introduce our notation, well-known results and basic concepts.
\Cref{sec:general position} presents our results for the special case of images in general position.
We relate these results to the existing literature on local upper bounds.
Then, in \Cref{sec:not general position}, we generalize these results.
Our novel algorithm is presented in \Cref{sec:alg}.  Finally, the efficacy of the parallelization achieved by PEA is demonstrated by the numerical results presented in  \Cref{sec: numerical study}.

\section{Preliminaries}
\label{sec:pre}
In this section, we introduce concepts from the field of multi-objective optimization that we use in this paper. For a more comprehensive introduction, please refer to the book by Ehrgott~\cite{Ehrgott2005}.
We consider multi-objective integer optimization problems
\[
\begin{array}{ll}\tag{\MOIP}
	\min &\; f(x)=(f_1(x),\dots,f_k(x))^\top \\
	\text{s.\,t.} &\; x\in X
\end{array}
\]
with $n\in\mathbb{N}$ variables and \emph{feasible set} $X\subseteq\mathbb{Z}^n$. The vector-valued objective function $f$ maps each feasible solution $x\in X$ to its \emph{image} $f(x)$.
The \emph{image set} $Y\coloneqq \{f(x) : x \in X\}\subseteq\mathbb{R}^k$ subsumes all possible images. The vector spaces $\mathbb{R}^n$ and $\mathbb{R}^k$ are called the \emph{decision space} and the \emph{image space}, respectively. We use the following notation: For any $y\in \mathbb{R}^k$ and $i\in\{1,\dots,k\}$, let $y_{-i}$ be the projection of $y$ onto $\mathbb{R}^{k-1}$ that excludes the $i$-th component, i.\,e., $y_{-i}\coloneqq (y_1,\dots,y_{i-1},y_{i+1},\dots,y_k)$. Furthermore, for any $i\in\{1,\dots,k\}$, we use $[i]\coloneq\{1,\dots,i\}$, $y_{[i]}=(y_1,\dots,y_{i})$ and $y_{-[i]}\coloneq (y_{i+1},\dots,y_k)$. For $a,b\in \mathbb{R}^k$, we denote the multi-dimension intervals by $[a,b]$, $(a,b)$, etc.

Since there is no canonical ordering in the image space $\mathbb{R}^k$, we utilize component-wise orders to define optimality: For images~\(y,\bar{y}\in\mathbb{R}^k\), the \emph{weak component-wise order}, the \emph{component-wise order}, and the \emph{strict component-wise order} are defined by 
\begin{align*}
	&y\leqq\bar{y} \;\text{if and only if}\; y_i \leq \bar{y}_i\; \text{for all}\; i\in[k], \\
	&y\leq\bar{y} \;\text{if and only if}\; y_i \leq \bar{y}_i\; \text{for all}\;  i\in[k] \; \text{and}\; y \neq \bar{y},\\
	&y<\bar{y} \;\text{if and only if}\; y_i < \bar{y}_i \; \text{for all}\;  i\in[k],
\end{align*}
respectively.
Then, a feasible solution~\(x^*\in X\) is called \emph{efficient} if there does not exist a feasible solution~\(x\in X\) such that \(f(x)\leq f(x^*)\).  The corresponding image \(y^*=f(x^*)\) is called \emph{nondominated}. The set of efficient solutions is called the \emph{efficient set} and denoted by \(X_E\). The set of nondominated images is called the  \emph{nondominated set} and denoted by \(Y_N\).  
In the following, we only consider \MOIP\ instances with finite nondominated set.

\subsection{The Lexicographic Epsilon-Constraint Scalarization}

In the epsilon-constraint scalarization, one of the objectives is minimized while all others are bounded from above and turned into constraints. It was first introduced by Haimes et al.~\cite{Haimes1971}. 

Optimal solutions of epsilon-constraint scalarization problems are not guaranteed to be efficient (they are only guaranteed to be \emph{weakly efficient}, cf.\ Ehrgott~\cite{Ehrgott2005}). In order to obtain efficient solutions of \MOIP, we employ the \emph{lexicographic epsilon-constraint scalarization}: For a given vector $\e\in\mathbb{R}^{k-1}$, the corresponding lexicographic epsilon-constraint scalarization \(\emethod{\e}\) is defined as 

\[
\begin{array}{ll}\tag{$\emethod{\e}$}
	\lexmin & (f_{k}(x),f_{k-1}(x),\dots,f_{1}(x))\\
	\text{s.\,t.} & f_{-k}(x)<\e,\\
	& x \in X. 
\end{array}
\]
We allow $\e_i=\infty$ for $i\in[k-1]$ to indicate that the corresponding objective is unconstrained. For clarity, we fix the ordering of the objectives. 
However, all results hold true for any ordering.
We refer to the image $y\in Y_N$ of an optimal solution $x\in X_E$  of $\emethod{\e}$ as $\emph{optimal}$ for $\emethod{\e}$.
For the remainder of this paper, we assume that we have a black-box solver for $\emethod{\e}$ at hand. 
If appropriate, such a black-box solver can be an integer programming solver like CPLEX~\cite{IBM2023} or Gurobi~\cite{GurobiOptimization2023}.
Note that the strict inequality constraints $f_{-k}(x)<\e$ can be simulated by the constraints $f_{-k}(x) \leqq \e - (\delta,\dots,\delta)$ for some $\delta> 0$. Such a $\delta$ is guaranteed to exist, since the nondominated set is required to be finite. Furthermore, $\delta=1$ is a viable choice if the objective functions assume only integer values. 

It is well-known (cf.\ Laumanns et al.~\cite{Laumanns2006}) that for lexicographic epsilon-constraint scalarization problems the optimal image is always uniquely optimal. In addition, for every nondominated image $y\in Y_N$, there exists an $\e\in\mathbb{R}^{k-1}$ such that $y$ is optimal for $\emethod{\e}$. Therefore, it is possible to obtain all nondominated images by repeatedly solving lexicographic epsilon-constraint scalarization problems for different values of $\e$.

It is important to note that solving a lexicographic epsilon-constraint scalarization problem may require more effort than solving an epsilon-constraint scalarization problem. Therefore, the scalarization problems that are solved by PEA might take longer than the ones of similar methods.
However, the benefit from the parallelization of PEA makes up for this additional effort.

\subsection{Epsilon-Components}
In this section, we introduce \emph{epsilon-components}, cf.~\cite{Prinz2024}, that we use to define an order on the parameters of the lexicographic epsilon-constraint scalarization:  Similarly to weight set components for weighted sum scalarization problems (cf.\ Przybylski et al.~\cite{Przybylski2010}), we associate each nondominated image $y$ with the set of all parameters $\e\in\mathbb{R}^{k-1}$ for which $y$ is optimal for $\emethod{\e}$.

\begin{definition}\label{def:epsilon components}
	For a nondominated image $y\in Y_N$, the \emph{epsilon-component} is defined as 
	$$E(y)\coloneq \left\{\e \in\mathbb{R}^{k-1} : y = \arglexmin_{\bar{y}\in Y_N} \{(\bar{y}_k,\dots,\bar{y}_1) : \bar{y}_{-k}<\e\}\right\}.$$
\end{definition}

The following four-objective example illustrates this concept.
\begin{example}\label{ex:introductory4D}
	We consider $f=\text{id}$ and
	\begin{align*}
		X=Y=\Big\{&y^1=(4,1,2,1)^\top, y^2=(2,4,3,2)^\top, y^3=(1,3,4,3)^\top\Big\}.
	\end{align*}
	Let us first consider the nondominated image $y^1$:
    This image is only feasible for $\emethod{\e}$ if $y^1_{-4}<\e$. Thus, 
	$$ \E{y^1}\subseteq \left\{\e\in\mathbb{R}^3 : y^1_{-4}<\e\right\}=\left\{\e\in\mathbb{R}^3 : (4,1,2)<\e\right\}. $$
	Furthermore, $y^1$ has the smallest fourth objective function value among all nondominated images. Thus, if $y^1$ is feasible for a lexicographic epsilon-constraint scalarization problem, then it is also optimal. Therefore, it holds that
	$$\E{y^1}=\left\{\e\in\mathbb{R}^3 : (4,1,2)<\e\right\}.$$
	
	Similarly, $y^2$ is only feasible for $\emethod{\e}$ if $y^2_{-4}<\e$. However, $y^1$ has a lower fourth objective function value than $y^2$, i.\,e., $y^1_4<y^2_4$. Consequently, if $y^1$ is feasible, then $y^2$ is not optimal. Thus,
	\begin{align*}
		\E{y^2}&=\left\{\e\in\mathbb{R}^3: y^2_{-4} < \e\right\}\backslash\E{y^1}\\
		&= \left\{\e\in\mathbb{R}^3: (2,4,3) < \e\right\} \cap \left\{\e\in\mathbb{R}^3 : (4,1,2)\nless\e\right\}.
	\end{align*}
	Analogously, for any $y\in Y_N$, we get 
	\begin{align} \label{eq:epsilon components} 
		\E{y} = \left\{ \e \in\mathbb{R}^{k-1} : y_{-k} < \e\right\} \backslash \bigcup_{\substack{\bar{y}\in Y_N\\ (\bar{y}_k,\dots,\bar{y}_1)\llex (y_k,\dots,y_1)}}  \E{\bar{y}}.\end{align}
  The epsilon-components from this example are depicted in \cref{fig:imgage set projection 4D}.
\end{example}
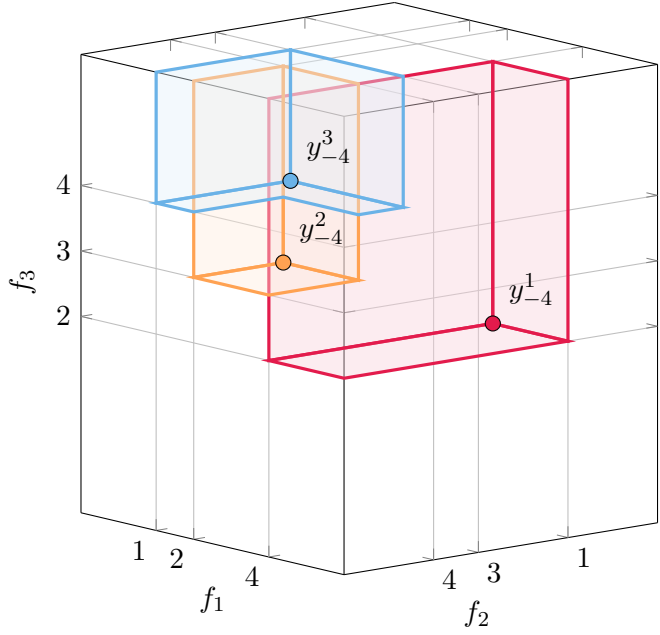
\begin{figure}
	\centering
		\subfloat{
		\begin{tikzpicture}[scale = 1]
			\begin{axis}[
				view={-50}{-10},
				width=260pt,
				height=260pt,	
				xmin=-1,xmax=6,
				ymin=-1,ymax=6,
				zmin=-1,zmax=6,
				grid=major,
				clip=true,
				xlabel={$f_1$},
				ylabel={$f_2$},
				zlabel={$f_3$},
                xtick ={1,2,4},
				ytick={1,3,4},
                ztick={2,3,4},
				every axis x label/.style={
				at={(axis cs: 2.5,6,-1.5)},
				anchor=north},
			every axis y label/.style={
				at={(axis cs: 6,2.5,-2)},
				anchor=east
			},
				]
				\cuboid{4}{1}{2}{6}{6}{6}{i1}{0.4}
				\cuboid{2}{4}{3}{4}{6}{6}{i2}{0.4}
				\filldraw[draw=i3,fill opacity=0.4,fill=i3!20,very thick] (axis cs: 1,3,4) -- (axis cs: 1,6,4) -- (axis cs: 2,6,4) -- (axis cs: 2,4,4)--    (axis cs: 4,4,4) -- (axis cs: 4,3,4)-- cycle;
				\filldraw[draw=i3,fill opacity=0.4,fill=i3!20,very thick] (axis cs: 1,3,4) -- (axis cs: 1,3,6) -- (axis cs: 1,6,6) -- (axis cs: 1,6,4)-- cycle;
				\filldraw[draw=i3,fill opacity=0.4,fill=i3!20,very thick] (axis cs: 1,3,4) -- (axis cs: 1,3,6) -- (axis cs: 4,3,6) -- (axis cs: 4,3,4)-- cycle;
				\node[draw,shape=circle,opacity=1,fill=i1, inner sep=2pt, label= above right : $y^1_{-4}$] at (axis cs:4,1,2) {};
				\node[draw,shape=circle,opacity=1,fill=i2, inner sep=2pt, label= above right : $y^2_{-4}$] at (axis cs:2,4,3) {};
				\node[draw,shape=circle,opacity=1,fill=i3, inner sep=2pt, label= above right : $y^3_{-4}$] at (axis cs:1,3,4) {};
			\end{axis}
			
	\end{tikzpicture}}

\caption{A visualization of \Cref{ex:introductory4D}: The projected nondominated images in $\mathbb{R}^3$  and their epsilon-components.} 
\label{fig:imgage set projection 4D}
\end{figure}
 
 We observe the following structural properties of epsilon-components. Similar properties (though in a slightly different context) are also given by Kirlik and Sayın~\cite{Kirlik2014}, among others. We restate the properties here in the context of epsilon-components and include a proof for the sake of completeness. We use $\cl$ to denote the closure of a set.
\begin{lemma} \label{lem: star shaped}
	Let $y\in Y_N$ and $\e\in\E{y}$. Then, the following holds:
	\begin{enumerate}[label=(\roman*)]
		\item For all $\e^*\in (y_{-k},\e]$, it holds that $\e^*\in\E{y}$. \label{item: star shaped 1}
		\item For all $\e^*\in [y_{-k},\e]$, it holds that $\e^*\in\cl{\E{y}}$. \label{item: star shaped 2}
	\end{enumerate}
\end{lemma}
\begin{proof} \phantom{linebreak} 		
	\begin{enumerate}[label=(\roman*)] 
		\item It holds that $y_{-k}<\e^*$. Thus, $y$ is feasible for $\emethod{\e^*}$. Furthermore, if some $y^*\in Y$ is feasible for $\emethod{\e^*}$, then it is also feasible for $\emethod{\e}$. Therefore, $\e\in\E{y}$ implies $(y_k,\dots,y_1)\llex(y^*_k,\dots,y_1^*)$. Hence, any image feasible for $\emethod{\e^*}$ has a worse lexicographic objective function value than $y$. Therefore, it is $\e^*\in\E{y}$.
		\item  For each $m \in\mathbb{N}$, we define $\e^{(m)}$ component-wise by 
		$$\e_i^{(m)}\coloneqq \begin{cases}\e^*_i+\frac{1}{m}, & \text{ if  }\e^*_i=y_i, \\ \e^*_i, &\text{else} \end{cases}, \text{ for } i\in[k-1].$$
		Then, for sufficiently large $m'$  it holds that $\e^{(m')}\in (y_{-k},\e]$. Therefore, by \ref{item: star shaped 1}, we obtain $\e^{(m)}\in\E{y}$ for all $m\ge m'$. It is  $\lim_{m\to\infty}\e^{(m)} = \e^*$, and, thus, $\e^*$ is an accumulation point of the sequence ${(\e^{(m)})}_{m\in\mathbb{N}}$. Therefore, it holds that $\e^*\in\cl{\E{y}}$.
	\end{enumerate}
\end{proof}

\subsection{Viable Combinations}
Next, we describe the set of all parameters for which we solve lexicographic epsilon-constraint scalarization problems to compute the nondominated set. Hereby, each parameter is defined by $k-1$ images. Furthermore, we use so-called dummy images to represent unconstrained objectives.
\begin{definition}
	For $t\in[k]$, the $t$-th \emph{dummy image} is given by
	$$d_i^t\coloneq \begin{cases} 
		\infty, & i=t \\
		-\infty, & i\ne t
	\end{cases}, \text{ for all } i\in[k].
	$$
\end{definition}
The dummy images can also be described by sufficiently small and large values, which is particularly useful for visualization purposes. 
The following observations are only for illustrative purposes and to give the reader an intuition. Thus, they are provided without proof.

Combining \Cref{eq:epsilon components} and \Cref{lem: star shaped}, we get that, for each $y\in Y_N$, there exist a certain number $\alpha(y)\in\mathbb{N}$ of parameters $\e^{(i)}\in(\mathbb{R}\cup \{\infty\})^{k-1}$ such that $\E{y}$ can be written as
$$\E{y} = \bigcup_{i=1}^{\alpha(y)}   (y_{-k},\e^{(i)}].$$ 
Additionally, for each $\e^{(i)}$ and $j\in[k-1]$, there exists an image $\bar{y}\in Y_N \cup \{d^j\}$ such that $\e^{(i)}_j=\bar{y}_j$. Thus, each of these parameters can be defined by $k-1$ images. 
Such an image $\bar{y}$ has a better lexicographic objective function value than $y$ and an ``adjacent'' epsilon-components, i.\,e., $\e^{(i)}\in\cl{\E{\bar{y}}}\cap\E{y}\ne \emptyset$.
Moreover, if we have found all the parameters $\e^{(i)}$ that describe the epsilon-components and solved a lexicographic epsilon-constraint scalarization problem for each, then we have explored the entire parameter space. That is, for each $\e\in(\mathbb{R}\cup \{\infty\})^{k-1}$ we know which nondominated image is optimal for $\emethod{\e}$. 

To summarize, we have a finite set of parameters we want to enumerate and for each such parameter $\e$ there exist $k-1$ images (nondominated or dummy) $y^1,\dots,y^{k-1}$ with
$$\e\in \clE{y^i}{i} \coloneqq \left\{\cl{\E{y^i}}: \e_i=y^i_i\right\}$$
for $i\in[k-1]$.

For any $\e$ in such a set $\clE{y^i}{i}$, it holds that the image $y^i$ itself is not feasible for $\emethod{\e}$. Thus, solving the corresponding scalarization problem yields another nondominated image. 
  
The set of parameters PEA enumerates is given as follows.
\begin{definition}
	Let $\Y^1,\dots,\Y^{k-1}\in Y_N\cup\{d^1,\dots,d^{k-1}\}$. Then, $\mathcal{\Y}=(\Y^1,\dots,\Y^{k-1})$ is a \emph{viable combination} of $Y_N$, if 
    $$\bigcap_{i=1}^{k-1} \clE{\Y^i}{i} \ne \emptyset.$$ 
    Each viable combination defines a \emph{viable parameter}
	$$ \e(\mathcal{Y})\coloneqq\left(\Y^1_1,\dots,\Y^{k-1}_{k-1}\right).$$
	We denote the set of all viable combinations of $Y_N$ by $\V{Y_N}$.
\end{definition}
For each viable combination $\Y$ it holds that that $\bigcap_{i=1}^{k-1} \clE{\Y^i}{i}=\{\e(\mathcal{Y})\}$, since any  $\e\in \bigcap_{i=1}^{k-1} \clE{\Y^i}{i}$ has $k-1$ fixed components.


We illustrate the concept of viable combinations and how they describe the parameter space/epsilon-components in the following example.

\begin{example}\label{ex:introductory_vc}
	Consider the image set of \Cref{ex:introductory4D}. We use the dummy images 
    $$d^1=(5,0,0,0)^\top,d^2=(0,5,0,0)^\top,d^3=(0,0,5,0)^\top \text{ and } d^4=(0,0,0,5)^\top.$$ Additionally, we  extend the concept of epsilon-components (cf.\ \Cref{def:epsilon components}) and also refer to epsilon-components of dummy images. Hereby, $\E{d^4}$ represents all parameters for which the lexicographic epsilon-constraint scalarization is infeasible.  There are eleven viable combinations with viable parameters as depicted in \Cref{fig:dummies and viable combinations}. Furthermore, we have
	\begin{align*}
	    &\E{y^1}=(y^1_{-4},\e(d^1,d^2,d^3)],\\
        &\E{y^2}=(y^1_{-4},\e(y^1,d^2,d^3)] \text{ and } \\
        &\E{y^3}=(y^3_{-4},\e(y^2,d^2,d^3)]\cup (y^3_{-4},\e(y^1,y^2,d^3)].
	\end{align*}
    $\E{d^4}$ can be described by the remaining seven viable parameters. Therefore, we need the viable combinations $(d^1,d^2,d^3)$, $(y^1,d^2,d^3)$, $(y^2,d^2,d^3)$ and $(y^1,y^2,d^3)$ to describe the epsilon-components of all nondominated images and the remaining seven viable combinations to describe $\E{d^4}$. 
\end{example}
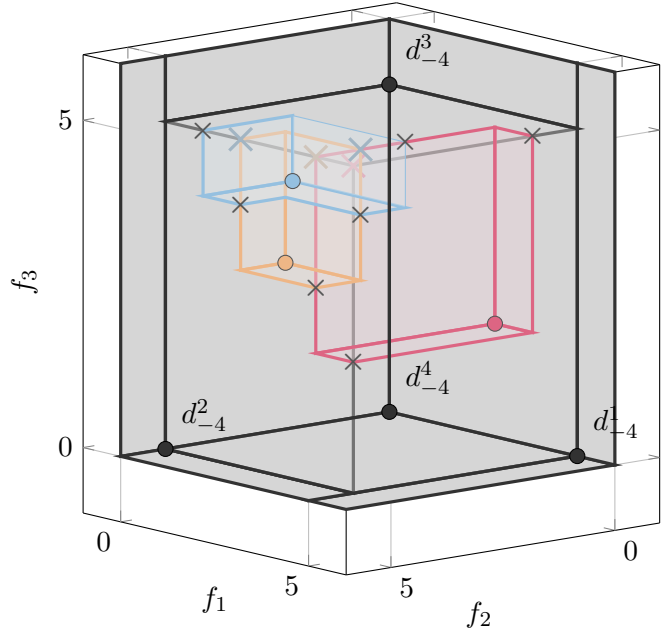
\begin{figure}[ht]
	\centering
	\begin{tikzpicture}[scale = 1]
		\begin{axis}[
			view={-50}{-10},
			width=260pt,
			height=260pt,	
			xmin=-1,xmax=6,
			ymin=-1,ymax=6,
			zmin=-1,zmax=6,
			grid=major,
			clip=true,
			xlabel={$f_1$},
			ylabel={$f_2$},
			zlabel={$f_3$},
            xtick ={0,5},
			ytick={0,5},
            ztick={0,5},
			every axis x label/.style={
				at={(axis cs: 2.5,6,-1.5)},
				anchor=north},
			every axis y label/.style={
				at={(axis cs: 6,2.5,-2)},
				anchor=east
			},
			]
			\cuboid{5}{0}{0}{6}{6}{6}{dummy}{1}
			\cuboid{0}{5}{0}{5}{6}{6}{dummy}{1}
			\cuboid{0}{0}{5}{5}{5}{6}{dummy}{1}
            \node[draw,shape=cross out,draw=i1,very thick] at (axis cs: 5,5,5 ) {};
			\cuboid{4}{1}{2}{5}{5}{5}{i1}{0.4}
            \node[draw,shape=cross out,draw=i2!60!black,very thick] at (axis cs: 4,5,5 ) {};
			\cuboid{2}{4}{3}{4}{5}{5}{i2}{0.4}
            
			\node[draw,shape=cross out,draw=i3!60!black,very thick] at (axis cs: 2,5,5 ) {};
			\node[draw,shape=cross out,draw=i3!60!black,very thick] at (axis cs: 4,4,5 ) {};
			\filldraw[draw=i3,fill opacity=0.4,fill=i3!20,very thick] (axis cs: 1,3,4) -- (axis cs: 1,5,4) -- (axis cs: 2,5,4) -- (axis cs: 2,4,4)--    (axis cs: 4,4,4) -- (axis cs: 4,3,4)-- cycle;
			
			\filldraw[draw=i3,fill opacity=0.4,fill=i3!20,very thick] (axis cs: 1,3,4) -- (axis cs: 1,3,5) -- (axis cs: 1,5,5) -- (axis cs: 1,5,4)-- cycle;
			\filldraw[draw=i3,fill opacity=0.4,fill=i3!20] (axis cs: 1,3,4) -- (axis cs: 1,3,5) -- (axis cs: 4,3,5) -- (axis cs: 4,3,4)-- cycle;
			\node at (axis cs: 1,5,5) {$\boldsymbol\times$};
			\node at (axis cs: 4,3,5) {$\boldsymbol\times$};
			\node at (axis cs: 4,4,4) {$\boldsymbol\times$};
			\node at (axis cs: 2,5,4) {$\boldsymbol\times$};
			\node at (axis cs: 4,5,3) {$\boldsymbol\times$};
			\node at (axis cs: 5,1,5) {$\boldsymbol\times$};
			\node at (axis cs: 5,5,2) {$\boldsymbol\times$};
            \node[draw,shape=circle,opacity=1,fill=i1, inner sep=2pt] at (axis cs:4,1,2) {};
				\node[draw,shape=circle,opacity=1,fill=i2, inner sep=2pt] at (axis cs:2,4,3) {};
				\node[draw,shape=circle,opacity=1,fill=i3, inner sep=2pt] at (axis cs:1,3,4) {};
			\cuboid{0}{0}{0}{5}{5}{5}{dummy}{0.4}
			\node[draw,shape=circle,opacity=1,fill=dummy, inner sep=2pt, label= above right : $d^4_{-4}$] at (axis cs:0,0,0) {};
			\node[draw,shape=circle,opacity=1,fill=dummy, inner sep=2pt, label= above right : $d^1_{-4}$] at (axis cs:5,0,0) {};
			\node[draw,shape=circle,opacity=1,fill=dummy, inner sep=2pt, label= above right : $d^2_{-4}$] at (axis cs:0,5,0) {};
			\node[draw,shape=circle,opacity=1,fill=dummy, inner sep=2pt, label= above right : $d^3_{-4}$] at (axis cs:0,0,5) {};
		\end{axis}
		
	\end{tikzpicture}
	
	\caption{The projected nondominated images and dummy images as well as their epsilon-components of \Cref{ex:introductory4D} in $\mathbb{R}^3$. All viable parameters are marked. The four viable parameters that describe the epsilon-component of a nondominated image are marked bigger and in their respective color. The other seven parameters describe the epsilon-component of the dummy image $d^4$, i.\,e., the part of the parameter space for which the respective epsilon-constraint scalarization problems are infeasible.}
	\label{fig:dummies and viable combinations}

\end{figure}
So far we have only given an intuition why viable combinations describe the entire parameter space. We give the formal proofs and show how $\V{Y_N}$ can be enumerated in \Cref{sec:general position} and \Cref{sec:not general position}. 

In \Cref{ex:introductory_vc}, each viable combination describes a different viable parameter. This is because the image set is in \emph{general position}, i.\,e., for all $i\in[k]$ and $y,\bar{y}\in Y_N$ with $y\ne\bar{y}$, it holds that  $y_i\ne\bar{y}_i$. Furthermore, each viable parameter is necessary to describe one epsilon-component. This is not the case if $Y_N$ is not in general position. 
In that case, it might happen that there are two  viable combinations $\Y,\Z\in\V{Y_N}$ with $\Y\ne\Z$, but $\e(\Y)=\e(\Z)$. Therefore, we first consider the simpler special case. 

\section{Viable Combinations under General Position}
\label{sec:general position}
In this section, we study properties of viable combinations and how they can be used to compute the nondominated set under the following assumption.

\begin{assumption}
	$Y_N$ is in general position.
\end{assumption}

This is a rather restrictive assumption in multi-objective optimization, but it makes the following proofs less technical. We show in \Cref{sec:not general position} that similar results still hold, when $Y_N$ is not in general position.

\begin{remark}
    Since we assume that $Y_N$ is in general position, for two nondominated images $y,\bar{y}$, it holds that
    $$(y_k,\dots,y_1)\llex (\bar{y}_k,\dots, \bar{y}_1) \text{ if and only if } y_k<\bar{y}_k.$$
    Therefore, a nondominated image $y$  that is feasible for some $\emethod{\e}$ is also optimal if no nondominated image $\bar{y}$ with $\bar{y}_k<y_k$ is feasible.

    Furthermore, no nondominated images share any objective function value. Therefore, for all $i\in[k]$ it is $y_i\ngtr\bar{y}_i$ if and only if $y_i<\bar{y}_i$.
\end{remark}

The theoretical results presented in this section form the foundation for the multi-objective PEA.
The approach can be outlined as follows:
We define a partial order on the set of all viable combinations such that it induces a directed tree. Each vertex represents a viable combination/parameter. Then, the nondominated image that is optimal for the corresponding lexicographic epsilon-constraint problem is used to explore outgoing arcs and obtain new viable combinations.

We show that, for every nondominated image, there exists a viable combination such that it is optimal for the corresponding lexicographic epsilon-constraint problem. Consequently, enumerating all viable combinations and solving a lexicographic epsilon-constraint problem for each is sufficient to compute the entire nondominated set.
Furthermore, under the general position assumption, we show that the parameters defined by viable combinations are projections of local upper bounds, cf.~Klamroth et al.~\cite{Klamroth2015}. 
Therefore,  the number of viable combinations is in $\mathcal{O}(|Y_N|^{\lfloor\frac{k}{2}\rfloor})$ (as the number of local upper bounds is known to be in $\mathcal{O}(|Y_N|^{\lfloor\frac{k}{2}\rfloor})$).

In the following, $\delta>0$ is the same $\delta$ that can be  used to turn strict inequality constraints in the definition of $\emethod{\e}$ into non-strict inequalities, i.\,e.,
$$\delta<\min_{\substack{y,\bar{y}\in Y_N\\ y\ne \bar{y}}} \min_{i\in [k]} |y_i-\bar{y}_i|.$$
We show two conditions that hold for any viable combination and the optimal image of the corresponding scalarization  problem. Afterwards, we use them to define an order on the set of viable combinations. 
The first condition gives a criterion that any viable combination must satisfy.
The second condition states, that the $k$-th objective function value of the optimal image is strictly larger than the $k$-th objective function value of all images that appear in the viable combination.
\begin{corollary}\label{cor:bascis}
	Let $\Y$ be a viable combination. Then, the following holds:
	\begin{enumerate}[label=(\roman*)]
		\item \label{item:basic 1} For all $i,j\in[k-1]$ with $i\ne j$, it holds that $\Y^j_i< \Y^i_i$.
		\item \label{item: basic 2} Let $y^*$ be the optimal image of $\emethod{\e(\Y)}$. Then, for any $i\in[k-1]$, it holds that $y^*_k>\Y^i_k$.
	\end{enumerate}
\end{corollary}
\begin{proof} \phantom{linebreak}
	
	\begin{enumerate}[label=(\roman*)]
		\item  Suppose there exist $i,j\in[k-1]$ with $i\ne j$ such that  $\Y^j_i> \Y^i_i$. 
        We show that this implies that $\e(\Y)\notin \clE{\Y^j}{j}$ which contradicts that $\Y$ is a viable combination.
        It holds that $\Y^j$ is not feasible for $\emethod{\e^*}$, for any $\e^*\in B_{\delta}(\e(\Y))$. This holds because for any such $\e^*$, it is $\Y^j_i>\Y^i_i+\delta=\e(\Y)_i+\delta \ge \e^*_i$. Thus, $B_{\delta}(\e(\Y))\cap\E{\Y^j}=\emptyset$, which contradicts $\e(\Y)\in \clE{\Y^j}{j}$. $B_{\delta}(\e(\Y))$ denotes the ${\delta}$-neighborhood of $\e(\Y)$.
		\item 
        For  $\e^*\in B_{\delta}(\e(\Y))$ it holds that $y^*$ is feasible for $\emethod{\e^*}$. Thus, for any $i\in[k-1]$ with $y^*_k<\Y^i_k$, we have $B_{{\delta}}(\e(\Y))\cap\E{\Y^i}=\emptyset$ which contradicts $\e(\Y)\in\clE{\Y^i}{i}$.
	\end{enumerate}
 
\end{proof}

We now define an order on the set of viable combinations. The idea is the following: Given a viable combination~$\mathcal{Y}\in\V{Y_N}$, we solve a lexicographic epsilon-constraint scalarization and obtain a new nondominated image $y^*$. Then, we only use this information to generate new viable combinations, i.\,e., information of $\Y^1,\dots,\Y^{k-1}$ and $y^*$. Therefore, we  generate new viable combinations by replacing any of the $\Y^i$ with $y^*$. 
The resulting image combination~$(\Y^1,\dots,\Y^{\ell-1},y^*,\Y^{\ell+1},\dots,\Y^k)$ cannot be a viable combination if \Cref{cor:bascis}.\ref{item:basic 1} does not hold. Therefore, we require $y^*_\ell > \Y^i_\ell$ for all $i\ne\ell$.
\begin{definition}\label{def:scions}
    Let $\mathcal{Y}\in\V{Y_N}$ and let $y^*$ be the optimal image of $\emethod{\e(\mathcal{Y})}$. Then, for every $\ell\in[k-1]$, we call the combination $(\Y^1,\dots,\Y^{\ell-1},y^*,\Y^{\ell+1},\dots,\Y^{k-1})$ the \emph{$\ell$-th scion} of $\mathcal{Y}$ if, for all $i\in[k-1]\backslash\{\ell\}$, it holds that $y^*_\ell > \Y^i_\ell$.
    We denote the $\ell$-th scion by $\scion{\Y}{\ell}$.
    Conversely, we call $\Y$ a \emph{precursor} of $\scion{\Y}{\ell}$.
\end{definition}
In view of \Cref{def:scions}, the necessary condition in~\Cref{cor:bascis}.\ref{item:basic 1} is also sufficient for  generating new viable combinations.

\begin{theorem}\label{thm:scions}
	Let $\mathcal{Y}\in\V{Y_N}$. Then, for all $\ell\in[k-1]$, it holds that the $\ell$-th scion, if it exists, is a viable combination.
\end{theorem}
\begin{proof}
    In the following, $y^*$ is the optimal image of $\emethod{\e(\Y)}$.
    
    To show that the $\ell$-th scion is a viable combination, we need to show that it is
    $$\bigcap_{i\in[k-1]\backslash\{\ell\}} \clE{\Y^i}{i} \cap \clE{y^*}{\ell}\ne\emptyset.$$
	To this end, we show that $$\e^*\coloneqq\left(\Y^1_1,\dots,\Y^{\ell-1}_{\ell-1},y^*_\ell,\Y^{\ell+1}_{\ell+1},\dots,\Y^{k-1}_{k-1}\right)\in \clE{\Y^i}{i}$$ for any $i\in[k-1]\backslash\{\ell\}$ by  contradiction. Showing that $\e^*\in \clE{y^*}{\ell}$ works analogously.

    Suppose $\e^*\notin\clE{\Y^i}{i}$.
    Then, it is $B_\delta(\e^*)\cap\E{\Y^i}=\emptyset$. In particular, it is
	$$\e^i\coloneq\left(\Y^1_1,\dots,\Y^{\ell-1}_{\ell-1},y^*_\ell,\Y^{\ell+1}_{\ell+1},\dots,\Y^{k-1}_{k-1}\right)+\delta e^i \notin \E{\Y^i}$$
    where $e^i$ denotes the $i$-th unit vector.
		
	By \Cref{cor:bascis}.\ref{item:basic 1} and since $y^*_\ell >\Y^i_\ell$, $\Y^i$ is feasible for $\emethod{\e^i}$.
	Since it is feasible but not optimal, there  exists a $\hat{y}$  with $\hat{y}_k<\Y^i_k$ that is feasible for $\emethod{\e^i}$.
    In addition, any image $y\ne\Y^i$ that is feasible for $\emethod{\e^i}$ is also feasible for $\emethod{\e(\Y)}$, i.\,e., $\hat{y}$ is feasible for $\emethod{\e(\Y)}$.
     By \Cref{cor:bascis}.\ref{item: basic 2}, it holds  that $\hat{y}^i_k<\Y^i_k<y^*_k$.
    This contradicts that $y^*$ is optimal for $\emethod{\e(\Y)}$. Consequently, $\e\in \clE{\Y^i}{i}$.
    
	Thus, it is $(\Y^1,\dots,\Y^{\ell-1},y^*,\Y^{\ell+1},\dots,\Y^{k-1})\in\V{Y_N}.$
\end{proof}
We can use the same approach as in the proof of \Cref{thm: scions not in general position} for statements of the following form:
\begin{corollary}\label{re:important stuff}
  Let $\e\in(\mathbb{R}\cup\{\infty\})^{k-1}$ such that there are images $y^1,\dots,y^{k-1}\in Y_N\cup \{d^1,\dots,d^{k-1}\}$ with $\e_i=y^i_i$ for all $i\in[k-1]$. Let $i\in[k-1]$ such that $y^i_j<y^j_j$ for all $j\in[k-1]\backslash\{i\}$, as is the case for a viable combination. Then, it holds that $\e\notin \clE{y^i}{i}$ implies that there exists a $y^*\in Y_N$ that is optimal for $\emethod{\e}$ with $y^*_k<y^i_k$. 
\end{corollary}

Next, we show that the scion order induces a directed tree that contains all viable combinations. The root of this tree is $(d^1,\dots,d^{k-1})$.
Therefore, we show that each viable combination, with the exception of $\vc{d}$, has exactly one precursor. 

The idea of the proof is the following: For a viable combination $\Z\in\V{Y_N}$, we search for another viable combination $\Y\in\V{Y_N}$ with $\Z=\scion{\Y}{\ell}$ for  some $\ell\in[k-1]$.
This means that the viable combinations $\Y$ and $\Z$ share $k-2$ images, i.\,e., for all $i\in[k-1]\backslash\{\ell\}$ it holds that $\Z^i=\Y^i$.
In addition, $\Z^\ell$ is the optimal image of $\emethod{\e(\Y)}$. Therefore, by \Cref{cor:bascis}.\ref{item: basic 2}, it holds that $\Z^\ell_k>\Y^i_k=\Z^i_k$.
That is, to find a precursor of $\Z$, we find the image with largest $k$-th objective function value and swap it for another image.
Hereby, the other image can be found by ``shooting'' a ray in direction $e^\ell$ until it ``hits'' the epsilon-component of an image. It is guaranteed to hit exactly one as $\E{d^\ell}$ lies in that direction ($\Z^\ell\ne d^\ell$  since $\Z\ne(d^1,\dots,d^k)$ and $d^\ell_\ell=\infty$) and the nondominated set is in general position.
\begin{theorem}\label{thm:at least one precursor}
	Let  $\Z\in\V{Y}$ with $\Z\ne\vc{d}$. Then, it is $$\left|\left\{\Y\in\V{Y_N}:\exists \ell\in[k-1] \text{ with } \Z=\scion{\Y}{\ell}\right\}\right|= 1.$$
\end{theorem} 
\begin{proof}
    We start by showing that 
    $$\left|\left\{\Y\in\V{Y_N}:\exists \ell\in[k-1] \text{ with } \Z=\scion{\Y}{\ell}\right\}\right|\ge 1.$$
	Let $\ell=\argmax_{i\in[k-1]}\{\Z^i_k\}$ and $$y(\Z)=\argmin_{y\in Y_N\cup\{d^\ell\}}\{y_\ell : (\Z^1_1,\dots,y_\ell,\dots,\Z^{k-1}_{k-1})\in \clE{y}{\ell} \text{ and } y_\ell>\Z_\ell^\ell)\}.$$ Note that $y(\Z)$ is well defined since $d^\ell$ satisfies 
 $$(\Z^1_1,\dots,d^\ell,\dots,\Z^{k-1}_{k-1})\in \clE{d^\ell}{\ell} \text{ and } d^\ell_\ell>\Z_\ell^\ell.$$ 
 Furthermore, as $Y_N$ is in general position, the minimum is unique. 
 We show that 
 $$\Y=(\Z^1,\dots,\Z^{\ell-1},y(\Z),\Z^{\ell+1},\dots,\Z^{k-1})$$ 
 is a viable combination and $\Z^\ell=\scion{\Y}{\ell}$.
 Therefore, we must show that 
 \begin{enumerate}
     \item $\Z^\ell$ is the optimal image of $\emethod{\e(\Y)}$ and
     \item $\Y$ is a viable combination.
 \end{enumerate}
	First, we show that $\Z^\ell$ is the optimal image of $\emethod{\e(\Y)}$.
 For the sake of contradiction, suppose that it is not. 
 Then, there exists a $y^*\in Y_N\backslash\{Z^\ell\}$ that is optimal for $\emethod{\e(\Y)}$ instead. Hence, it holds that $y^*_\ell<y(\Z)_\ell$, $y^*_i<\Z^i_i$ for all $i\in[k-1]\backslash\{\ell\}$ and $y^*_k<\Z^\ell_k$. 
 Consequently, $\e(\Z)\in\clE{\Z^\ell}{\ell}$ (which holds because $\Z\in\V{Y_N}$) implies that $y^*_\ell>\Z^\ell_\ell$.
 In addition, by \Cref{lem: star shaped}.\ref{item: star shaped 2}, it holds that $$(\Z^1_1,\dots,\Z^{\ell-1}_{\ell-1},y^*_\ell,\Z^{\ell+1}_{\ell+1}\dots,\Z^{k-1}_{k-1})\in \clE{y^*}{\ell}$$ and, thus,   
	$$y^*\in\left\{ y\in Y_N\cup\{d^\ell\}: (\Z^1_1,\dots,y_\ell,\dots,\Z^{k-1}_{k-1})\in \clE{y}{\ell} \text{ and } y_\ell>\Z_\ell^\ell\right\}.$$ This contradicts the choice of $y(\Z)$ since $y^*_\ell<y(\Z)_\ell$. 

	Next, we show that $\Y$ is a viable combination.
    Suppose this is not the case.
    For $y(\Z)$, that $\e(\Y)\in\clE{y(\Z)}{\ell}$ follows directly from 
    $$y(\Z)\in \{y\in Y_N \cup \{d^\ell\} : (\Z^1_1,\dots,y_\ell,\dots,\Z^{k-1}_{k-1})\in \clE{y}{\ell} \text{ and } y_\ell>\Z_\ell^\ell)\}. $$ 

    Hence, there exists an $i\in[k-1]\backslash\{\ell\}$ with $\e(\Y)\notin\clE{\Z^i}{i}$. 
	
		\Cref{re:important stuff} implies that there exists a $y^*\in Y_N$ with $y^*_k<\Z^i_k<\Z^\ell_k$ that is feasible for $\emethod{\e(\Y)}$. Hence, $\Z^\ell$ is not optimal which is a contradiction.
    
	Thus, it holds that $\Y$ is a viable combination and, by \Cref{def:scions},  $\Z=\scion{\Y}{\ell}$.

    It remains to be shown that
    $$\left|\left\{\Y\in\V{Y_N}:\exists \ell\in[k-1] \text{ with } \Z=\scion{\Y}{\ell}\right\}\right|\le 1.$$
	Again, let $\ell=\argmax_{i\in[k-1]}\{\Z^i_k\}$. Then, by \Cref{cor:bascis}.\ref{item: basic 2}, for any precursor $\U$ of $\Z$ and all $i\in[k-1]\backslash\{\ell\}$, it holds that $\U^i=\Z^i$. Let $$\W,\U\in\left\{\Y\in\V{Y_N}:\exists \ell\in[k-1] \text{ with } \Z=\scion{\Y}{\ell}\right\}.$$ Suppose $\W\ne\U$, i.\,e., $\W^i=\U^i=\Z^i$ for $i\in[k-1]\backslash\{\ell\}$ and $\W^\ell\ne \U^\ell$. W.\,l.\,o.\,g., $\W^\ell_\ell<\U^\ell_\ell$. But then, by \Cref{cor:bascis}, $\W^\ell$ is feasible for $\e(\U)$ which implies that $\Z^\ell$ is not optimal. Hence, $\Z$ is not a scion of $\U$ which is a contradiction.
\end{proof}

Thus, the order as defined in \Cref{def:scions} induces a directed tree with viable combinations as vertices and root $\vc{d}$.
The induced tree for the instance from  \Cref{ex:introductory4D} is visualized in \Cref{fig:vc tree}.
\begin{figure}
	\centering
		\begin{tikzpicture}[scale = 1,node distance=1cm and 0.5cm]
			\node[draw,very thick] (A) at (0,0) {$(d^1,d^2,d^3)$};
			\node[draw, below left = of A,very thick] (B)  {$(y^1,d^2,d^3)$};
			\node[draw, below  = of A,very thick] (C)  {$(d^1,y^1,d^3)$};
			\node[draw, below right = of A,very thick] (D)  {$(d^1,d^2,y^1)$};
			\draw[->,very thick] (A) -- (B);
			\draw[->,very thick] (A) -- (C);
			\draw[->,very thick] (A) -- (D);
			\node[draw, below left = of B,very thick] (E)  {$(y^2,d^2,d^3)$};
			\node[draw, below  = of B,very thick] (F)  {$(y^1,y^2,d^3)$};
			\node[draw, below right = of B,very thick] (G)  {$(y^1,d^2,y^2)$};
			\draw[->,very thick] (B) -- (E);
			\draw[->,very thick] (B) -- (F);
			\draw[->,very thick] (B) -- (G);
			\node[draw, below left = of E,very thick] (H)  {$(y^3,d^2,d^3)$};
			\node[draw, below = of E,very thick] (I)  {$(y^2,d^2,y^3)$};
			\draw[->,very thick] (E) -- (H);
			\draw[->,very thick] (E) -- (I);
			\node[draw, below = of F,very thick] (J)  {$(y^1,y^3,d^3)$};
			\node[draw, below right = of F,very thick] (K)  {$(y^1,y^2,y^3)$};
			\draw[->,very thick] (F) -- (J);
			\draw[->,very thick] (F) -- (K);
		\end{tikzpicture}
	\caption{The tree as induced by the order described in \Cref{def:scions}, i.\,e., $G=(V,A)$ where $V=\V{Y_N}$ and $A=\{(\Y,\Z): \Z=\scion{\Y}{\ell} \text{ for some } \ell\in[k-1]\}$, for the nondominated set given in \Cref{ex:introductory4D}.}
	\label{fig:vc tree}
\end{figure}
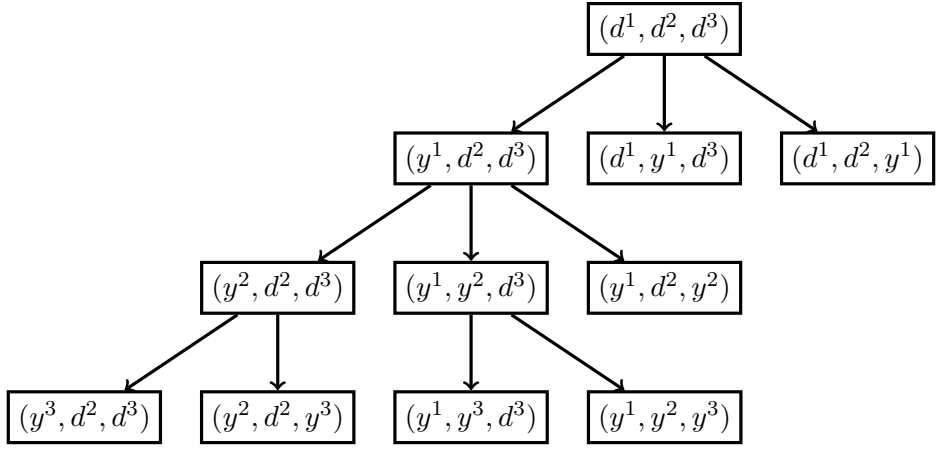

We now show that for each nondominated image $y$, there exists at least one viable combination $\Y$ such that $y$ is optimal for $\emethod{\e(\Y)}$. Thus, if we enumerate all viable combination and solve a lexicographic epsilon-constraint scalarization problem for each, we are guaranteed to compute the entire nondominated set.

The proof of the following theorem is visualized in \Cref{fig: finding vc}.
 We start with a parameter for which we know that $y$ is optimal and --- figuratively speaking --- iteratively shoot rays in the directions $e^1,\dots,e^{k-1}$. Each time, until we hit the epsilon component of an image. The respective images then constitute a viable combination $\Y(y)$ and $y$ is optimal for $\emethod{\e(\Y)}$.
 In addition, for all $i\in[k-1]$, it holds that $\Y(y)^i_{-[i]}<y_{-[i]}$.
 We later show that $\Y(y)$ is the only viable combination that has $y$ as optimal image and satisfies this. Therefore, this adds a nice criterion  when to save nondominated images  without ever having to check for duplicates.
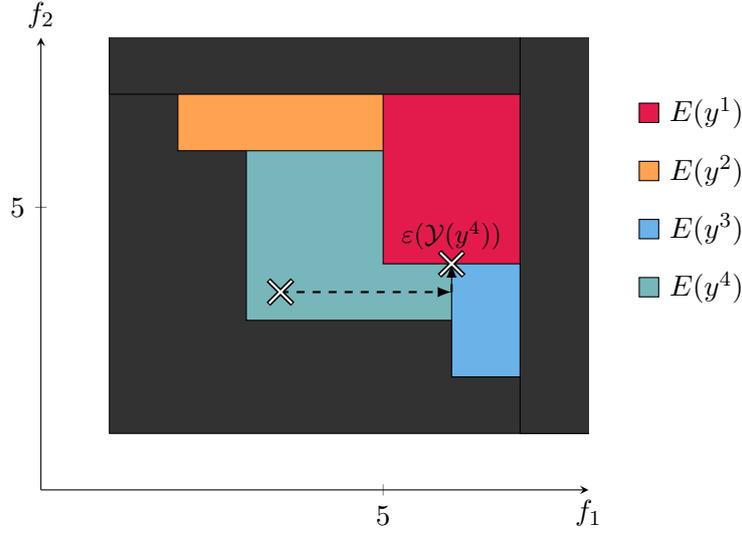
\begin{figure}[ht]
\centering
			\begin{tikzpicture}[scale = 1]
			
			\begin{axis}[
				axis lines=left,
				width=250pt,
				xmin=0,xmax=8,
				ymin=0,ymax=8,
				xtick ={5},
				ytick={5},
				xlabel={$f_1$},
				ylabel={$f_2$},	every axis x label/.style={
					at={(axis cs: 8,0,0)},
					anchor=north},
				every axis y label/.style={
					at={(axis cs: 0,8,0)},
					anchor=south
				}]
				\filldraw [fill=dummy,draw=black] (axis cs: 1,1) rectangle (axis cs:19,10);
                \filldraw [fill=i4,draw=black] (axis cs: 3,3) rectangle (axis cs: 19,10);
				\filldraw [fill=i3,draw=black] (axis cs: 6,2) rectangle (axis cs: 19,10);
				\filldraw [fill=i2,draw=black] (axis cs: 2,6) rectangle (axis cs: 19,10);
	
				\filldraw [fill=i1,draw=black] (axis cs: 5,4) rectangle (axis cs: 19,10);	
				\filldraw [fill=dummy,draw=black] (axis cs: 1,7) rectangle (axis cs: 19,10);
				\filldraw [fill=dummy,draw=black] (axis cs: 7,1) rectangle (axis cs: 19,10);
                \addplot[mark=x,mark options={scale=3.2}, line width=2pt] coordinates {(3.5,3.5)} ;
				\addplot[mark=x,white,mark options={scale=3},thick] coordinates {(3.5,3.5)} ;;
                \draw[-latex,dashed,thick] (axis cs: 3.5,3.5) -- (axis cs: 6,3.5);
                \draw[-latex,dashed,thick] (axis cs: 6,3.5) -- (axis cs: 6,4);

                \addplot[mark=x,mark options={scale=3.2}, line width=2pt] coordinates {(6,4)} ;
				\addplot[mark=x,white,mark options={scale=3},thick] coordinates {(6,4)} ;;
                \node[] at (axis cs: 6,4.5) {\small{$\e(\Y(y^4))$}};
			\end{axis}
     \node[rectangle,fill=i1,draw=black] (y1) at (8,5){};
		\node[right = of y1,xshift=-1cm] {$E(y^1)$};
		
		\node[rectangle,fill=i2,draw=black, below= of y1,yshift=0.5cm] (y2) {};
		\node[right = of y2,xshift=-1cm] {$E(y^2)$};
		
		\node[rectangle,fill=i3,draw=black, below= of y2,yshift=0.5cm] (y3) {};
		\node[right = of y3,xshift=-1cm] {$E(y^3)$};
		
		\node[rectangle,fill=i4,draw=black, below= of y3,yshift=0.5cm] (y4) {};
		\node[right = of y4,xshift=-1cm] {$E(y^4)$};
	\end{tikzpicture}

	\caption{We consider a tri-objective problem for which the nondominated set is given by $Y_N = \left\{ y^1= (5,4,1)^\top, y^2=(2,6,2)^\top, y^3=(6,2,4)^\top, y^4=(3,3,5)^\top\right\}.$ Then, the construction of $\Y(y^4)$ as in the proof of \Cref{thm: viable combination construction} can be interpreted as follows: Starting with $(y^4_1+\delta,y^4_2+\delta)$, we shoot a ray in direction $(1,0)$ until we ``hit'' the first epsilon-component, $\E{y^3}$. Then,  we ``shoot'' another ray in direction $(0,1)$ until we hit another epsilon component, $\E{y^1}$. Thus, $\Y(y^4)=(y^3,y^1)$.}
	\label{fig: finding vc}

\end{figure}
\begin{theorem}\label{thm: viable combination construction}
	Let $y\in Y_N$. Then, there exist $\Y(y)^1,\dots,\Y(y)^{k-1}\in Y_N\cup\{d^1,\dots,d^{k-1}\}$ such that the following hold:
	\begin{enumerate}[label=(\roman*)]
		\item \label{item: viable combination construction 1} For all $i\in[k-1]$, it holds that $\Y(y)^i_{-[i]}<y_{-[i]}$.
        \item \label{item: viable combination construction 3} The optimal image of $\emethod{\e(\Y(y))}$ is $y$.
		\item $\Y(y)=(\Y(y)^1,\dots,\Y(y)^{k-1})$ is a viable combination.
		
	\end{enumerate}
\end{theorem}
\begin{proof}
	We iteratively define images $\Y(y)^1,\dots,\Y(y)^{k-1}$:	
    For $i\in[k-1]$, we define
	$$\Li^i(y)\coloneq \{ \bar{y}\in Y_N \cup \{d^i\}: y_i<\bar{y}_i,  (\Y(y)^1_1,\makebox[1.2em][c]{.\hfil.\hfil.},\Y(y)^{i-1}_{i-1},\bar{y}_i,y_{i+1}+{\delta},\makebox[1.2em][c]{.\hfil.\hfil.},y_{k-1}+{\delta})\in \clE{\bar{y}}{i}\}$$
	and we set
	$$\Y(y)^i \coloneqq \argmin_{\bar{y}\in \Li^i(y)} \{\bar{y}_i\}.$$
	   For all $i\in[k-1]$ it holds that $d^i\in \Li^i(y)$. Thus, all $\Y(y)^i$ are well defined and, since we assume that $Y_N$ is in general position, also unique. In addition, from the construction of the $\Y(y)^i$ it follows that $\Y(y)^i_j<\Y(y)^j_j$ for all $i,j\in[k-1]$ with $j\ne i$. 
	\begin{enumerate}[label=(\roman*)]
		\item If $\Y(y)^i_j>y_j$  for some $i<j\le k$, then  $$(\Y(y)^1_1,\dots, \Y(y)^i_i,y_{i+1}+\delta,\dots,y_{k-1}+\delta)\notin \clE{\Y(y)^i}{i}$$  which contradicts $\Y(y)^i\in\Li^i(y)$.
        \item By construction, $y$ is feasible for $\emethod{\e(\Y(y))}$.
        Suppose it is not optimal. Thus, there exists a $y^*\in Y_N$ with $y^*_i<\Y(y)^i_i$ for all $i\in[k-1]$ and $y^*_k<y_k$. Furthermore, as $y\in Y_N$, there exists $J\subseteq[k-1]$ with $y^*_j>y_j$ for all $j\in J$. Let $j=\max J$. Then, we have that $\e(\Y(y))\in\E{y^*}$ and $$(\Y(y)^1_1,\dots,\Y(y)^{j-1}_{j-1},y^*_j,y_{j+1}+\delta,\dots,y_{k-1}+\delta)\in[y^*,\e(\Y(y)].$$ Therefore, by \Cref{lem: star shaped}.\ref{item: star shaped 2}, it is 
		$$(\Y(y)^1_1,\dots,\Y(y)^{j-1}_{j-1},y^*_j,y_{j+1}+\delta,\dots,y_{k-1}+\delta)\in \clE{y^*}{j}$$ and $y^*\in\Li^j(y)$ which contradicts the choice of $\Y(y)^j$.
		\item Suppose $\Y(y)$ is not a viable combination. Thus, for some $j\in[k-1]$, it holds that $\e(\Y(y))\notin\clE{\Y(y)^j}{j}$. Hence, by \Cref{re:important stuff}, there exists a $y^*\in Y_N$ with $y^*_k<\Y(y)^j_k$ and $y^*_i<\Y(y)^i_i$ for all $i\in[k-1]$ that is optimal for $\emethod{\e(\Y(y))}$. 
		Thus, by \ref{item: viable combination construction 3}, it holds that $y=y^*$. But this implies that 
		$$(\Y(y)^1_1,\dots,\Y(y)^j_j,y_{j+1}+\delta,\dots, y_{k-1}+\delta) \notin \clE{\Y(y)^j}{j}$$
		since $y$ is feasible for $\emethod{\e}$ for all $$\e\in B_{\frac{\delta}{2}}((\Y(y)^1_1,\dots,\Y(y)^j_j,y_{j+1}+\delta,\dots, y_{k-1}+\delta))$$ 
		and $y_k=y^*_k< \Y(y)^j_k$.
		Thus, $\Y(y)^j\notin \Li^j(y)$ which is a contradiction.
		
	\end{enumerate}
	
\end{proof}

We show that $\Y(y)$ is the only viable combination that satisfies Condition~\ref{item: viable combination construction 1} and \ref{item: viable combination construction 3} of \Cref{lem: star shaped}. 

\begin{corollary}\label{cor:unique vc}
	Let $y\in Y_N$. Then, there exists exactly one $\Y\in\V{Y_N}$ such that $y$ is optimal for $\emethod{\e(\Y)}$ and $\Y^i_{-[i]}<y_{-[i]}$ for all $i\in[k-1]$.
\end{corollary}
\begin{proof}
	Let $\Y(y)$ be the viable combination constructed in the proof of \Cref{thm: viable combination construction}. We have already shown that $y$ is optimal for $\emethod{\e(\Y(y))}$ and that $\Y(y)^i_{-[i]}<y_{-[i]}$ for all $i\in[k-1]$. 
	Let $\Z$ be a viable combination such that  $y$ is optimal for $\emethod{\e(\Z)}$ and $\Z^i_{-[i]}<y_{-[i]}$ for all $i\in[k-1]$. We show that $\Z^i=\Y(y)^i$ for all $i\in[k-1]$ by induction.
	\begin{itemize}
		\item $i=1$: We have that $(\Z^1_1,y_2+{\delta},\dots,y_{k-1}+{\delta})\in [\Z^1_{-k},\e(\Z)]$. Thus, \Cref{lem: star shaped}.\ref{item: star shaped 2} states that $(\Z^1_1,y_2+{\delta},\dots,y_{k-1}+{\delta})\in\clE{\Z^1}{1}$. Therefore, $\Z^1\in\Li^1(y)$. Suppose $\Z^1\ne \Y(y)^1$. Thus, $\Z^1_1>\Y(y)^1_1$. Since $\Y(y)^1_{-1}<y_{-1}$ and $y$ is feasible for $\emethod{\e(\Z)}$, it holds that $\Y(y)^1$ is feasible as well. Hence, as $y_k>\Y(y)^1_k$ by \Cref{cor:bascis}.\ref{item: basic 2}, $y$ cannot be optimal which is a contradiction. Thus, it holds that $\Y(y)^1=\Z^1$.
		\item $i>1$: By induction, for all $j<i$ it holds that $\Z^j=\Y(y)^j$ . Furthermore, since $\Z$ is a viable combination it holds that $\Z^i_j< \Z^j_j= \Y(y)^j_j$. Thus, $$(\Y(y)_1,\dots,\Y(y)_{i-1},\Z^i_i,y_{i+1}+{\delta},\dots,y_{k-1}+{\delta})\in [\Z^i_{-k},\e(\Z)].$$ Consequently, again by \Cref{lem: star shaped}.\ref{item: star shaped 2}, it is $\Z^i\in\Li^i(y)$. Combining this with $\Z^i\ne \Y(y)^i$ implies $\Y(y)^i_i<\Z^i_i$. Thus, since $\Y^i_{-[i]}<y_{-[i]}$, $\Y(y)^i$ is feasible for $\emethod{\e(\Z)}$ which, again, contradicts that $y$ is optimal. Thus, it holds that $\Z^i=\Y(y)^i$.
	\end{itemize}
\end{proof}

Therefore, we already have the basic outline of PEA: We start with the viable combination $(d^1,\dots,d^{k-1})$. Then, we solve the corresponding scalarization, generate scions and repeat. We only store optimal images if \Cref{cor:unique vc} is satisfied. Thus, each nondominated image is stored exactly once, and we do not need to check for duplicates. What remains to be shown is that the cardinality of the set of all viable combinations is in $\mathcal{O}(|Y_N|^{\lfloor\frac{k}{2}\rfloor})$. To this end, we link the set of viable combinations to the upper bound set of $Y_N$.

\subsection{Viable Combinations and Upper Bounds}\label{sec:upper bounds}
In this section, we explore the connection  between viable combinations and the representation of the  \emph{search region} - the subset of the image space that potentially contains further nondominated images, by so-called upper bounds, first introduced by Przybylski et al.~\cite{Przybylski2010a}.  

Formally, for a set $N\subseteq Y_N$, the search region is given by
$$
S(N)\coloneqq \left\{z \in\mathbb{R}^k: y \nleqq z \text{ for all } y\in N\right\}.
$$
Then, local upper bounds that describe the search region can be defined as follows.
\begin{definition}[Klamroth et al.~\cite{Klamroth2015}]\label{cor:upper bounds}
	Let $N\subseteq Y$ be a finite set of images and $u\in\mathbb{R}^k$ satisfy the following:
	\begin{enumerate}[label=(\roman*)]
		\item There is no $y\in N$ such that $y<u$.
		\item\label{item:upper bounds 2} There exist $k$ points of $N\cup\{d^1,\dots,d^{k}\}$, denoted by $y^1(u),\dots,y^k(u)$, such that $\begin{cases}
			y^i_i(u)=u_i\\
			y^i_{-i}(u)<u_{-i}
			
		\end{cases}, i=1,\dots,k.$
	\end{enumerate}
	Then, $u$ is a \emph{local upper bound} with respect to $N$. The set of all local upper bounds with respect to $N$ is denoted by $U(N)$. The images $y^1(u),\dots,y^k(u)$ are called the \emph{defining points} of $u$.
\end{definition}
Every local upper bound $u\in U(N)$ defines a \emph{search zone} $C(u)\subseteq \mathbb{R}^k$:
$$C(u)=\{z\in\mathbb{R}^k: z < u\}$$
and the search region can be written as 
$$S(N)=\bigcup_{u\in U(N)} C(u).$$
An upper bound based method to compute the nondominated set works as follows. Starting with $N=\emptyset$, such a method iteratively picks a local upper bound  $u\in U(N)$  and explores the associated search zone $C(u)$ with an appropriate scalarization problem. Then, if a new nondominated image $y$ is found, the updated local upper bound set $U(N\cup \{y\})$  is computed and $y$ is added to $N$. Otherwise, $C(u)$ is empty. Hereby, the update step is of crucial importance. Klamroth et al.~\cite{Klamroth2015} propose two different update strategies, one based on redundancy elimination and the other one based on redundancy avoidance. The latter uses defining points: When a new nondominated image $y$ is identified $u$ is replaced by new local upper bounds that are generated by verifying \Cref{cor:upper bounds}.\ref{item:upper bounds 2}. The search zones given by upper bounds in $U(N)$ intersect. Hence, there might be another $u'\in U(N)$ such that $y\in C(u')$. 
In this case, $u'\notin U(N\cup \{y\})$ and it needs to be replaced as well.


 D{\"a}chert et al.~\cite{Daechert2017} propose the following neighborhood structure among local upper bound to efficiently identify the $u'\in U(N)$ with $y\in C(u')$.
\begin{definition}[D{\"a}chert et al.~\cite{Daechert2017}]\label{def:upper bounds neighborhood}
    Let $u,u'\in U(N)$. Then, $u$ and $u'$ are neighbors if they share $k-1$ defining points, i.\,e., there exist $\ell,j\in[k]$ with $\ell\ne j$ such that for all $i\in[k]\backslash\{\ell,j\}$ it holds that $y^i(u)=y^i(u')$ and $y^j(u)=y^\ell(u')$. Then, $u$ is the $j$-neighbor of $u'$ and $u'$ is the $\ell$-neighbor of $u$.
    \end{definition}

The literature on computational geometry offers some insight on the cardinality of $U(Y_N)$: From Boissonnant et al.~\cite{Boissonnat1998} and Bringmann~\cite{Bringmann2013a} it can be derived that $|U(Y_N)|\in\mathcal{O}(|Y_N|^{\lfloor \frac{k}{2} \rfloor})$. We show that we can map viable combinations onto $U(Y_N)$ via an injective function $g$. This provides insight into the cardinality of $\V{Y_N}$ and relates the scion order to the order described in \Cref{def:upper bounds neighborhood}.

\begin{theorem}\label{thm:map to upper bounds}
	Let $\Y\in \V{Y_N}$ and let $\Y^k$ be the optimal image of $\emethod{\e(\Y)}$. Then, $(\Y^1_1,\dots,\Y^k_k)\in U(Y_N)$.
\end{theorem}
\begin{proof}
	Any $y<u\coloneqq(\Y^1_1,\dots,\Y^k_k)$ would be feasible for $\emethod{\e(\Y)}$ with $y_k<\Y^k_k$ which contradicts that $\Y^k$ is optimal for $\emethod{\e(\Y)}$. Furthermore, by \Cref{cor:bascis},  $\Y^1,\dots,\Y^k$ satisfy $$\begin{cases}
		\Y^i_i=u_i\\
		\Y^i_{-i}<u_{-i}
	\end{cases}, i=1,\dots,k.$$ Thus, by \Cref{cor:upper bounds}, $u\in U(Y_N)$. 
\end{proof}
\Cref{thm:map to upper bounds} implies that the function
$$g: \V{Y_N} \to U(Y_N), \Y \mapsto (\Y^1_1,\dots,\Y^k_k),$$
where $\Y^k$ is the optimal image of $\emethod{\e(\Y)}$, is injective. Therefore, $|\V{Y_N}|\le |U(Y_N)|$ and $|\V{Y_N}|\in\mathcal{O}(|Y_N|^{\lfloor\frac{k}{2}\rfloor})$. In addition, by \Cref{thm:map to upper bounds}, the viable parameters are projections of the local upper bounds w.\,r.\,t.\ $Y_N$ and the corresponding viable combination corresponds to the defining points $y^1(u),\dots,y^{k-1}(u)$. Furthermore, for two viable combinations $\Y,\Z$ it holds that $\Y$ is the $\ell$-th scion of $\Z$, if and only if $g(\Y)$ is the $\ell$-neighbor of $g(\Z)$ and $g(\Z)$ is the $k$-neighbor of $g(\Y)$. Thus, the tree induced by the scion order is equivalent to a subgraph of the graph induced by neighborhood structure among local upper bound w.\,r.\,t.\ \Cref{def:upper bounds neighborhood}. 

However, while D{\"a}chert et al.~\cite{Daechert2017} iteratively update the neighborhood structure and utilize it to efficiently update the search region, PEA directly enumerates the nondominated images according to a subtree of the graph induced by the neighborhood structure of  the local upper bounds set of the entire nondominated set, i.\,e., the neighborhood structure once all nondominated images have already been found.

In addition, while the generation of new local upper bounds and new viable combinations is similar there is a key difference: Upper bound based methods use parameters  in $\mathbb{R}^k$ to solve scalarization problems. Meanwhile, PEA works with parameters in $\mathbb{R}^{k-1}$.
Let $u$ be a local upper bound w.\,r.\,t.\ some $N\subseteq Y_N$. When a new image $y$ is discovered in $C(u)$ it means that no other nondominated image $\bar{y}$ with $\bar{y}\ge y$ exists. Thus, no more scalarization problems need to be solved for any parameter $u'> y$. That is, all $u'\in U(N)$ with  $u'> y$ need to be identified and replaced with local upper bounds w.\,r.\,t $N\cup\{y\}$.
In contrast, let $\Y$ be a viable combination of $Y_N$. When PEA solves $\emethod{\e(\Y)}$ and discovers a nondominated image it is not necessarily new, i.\,e., it might have been computed before. In addition, it only follows that $y$ is optimal for all parameters in $(y_{-k},\e(\Y)]$. That is, for all other viable combinations, solving a scalarization problem for the respective viable parameter might still yield a new nondominated images and PEA only generates new viable combinations and does not replace any.

Thus, working with viable combinations is more suited for parallelization, or at least, requires less communication between different threads: The viable combinations can be processed independently while the local upper bounds cannot.

\section{Viable Combinations in the Generic Case }
\label{sec:not general position} 
In \cref{sec:general position}, we need the property that the nondominated set is in general position.
However, while this assumption greatly helps in simplifying technical proofs, it is not a necessary condition for PEA\@.
In this section, we consider the case where the nondominated set $Y_N$ is not in general position. The idea is the following: We construct a related nondominated set $\Phi(Y_N)$ that is in general position. Hereby, $\Phi : Y_N\to \mathbb{R}^k$ is a function that slightly modifies each image. Since $\Phi(Y_N)$ is in general position, all results from \Cref{sec:general position}
 hold. For each viable combination $\Phi(\Y)$ of $\Phi(Y_N)$, we consider the preimage $\Y$. We show that it is a viable combination of $Y_N$ and an image $y$ is optimal for $\emethod{\e(\Y)}$ if and only if $\Phi(y)$ is optimal for $\emethod{\e(\Phi(\Y))}$. Thus, since the enumeration of all viable combinations $\Phi(Y_N)$ and solving a scalarization problem is sufficient to compute $\Phi(Y_N)$, it is also sufficient to enumerate the preimages to compute $Y_N$.

 The set $\Phi(Y_N)$ is only a means to an end and only used as a tool for the proofs in this section. It can be constructed as follows:
Let $K\coloneqq |Y_N|$. We assume that the nondominated images are ordered such that 
\begin{align}\label{eq:order}
	(y_k^1,\dots,y^1_1)\llex (y_k^2,\dots,y^2_1)\llex \dots \llex (y^K_k,\dots,y^K_1).
\end{align}
We define $\Phi(Y_N):=\{\Phi(y^1),\dots,\Phi(y^K)\}$ where for all $i\in[K]$ and $j\in[k]$ we set
$$
\Phi(y^i)_j\coloneqq\begin{cases} y^i_j + \frac{i\delta}{K+1},& \text{ if there exists a } s<i \text{ with } y^s_j=y^i_j, \\
	y^i_j, & \text{ else.} \end{cases}
$$

Under $\Phi$, all nondominated images remain nondominated and the order as given in \Cref{eq:order} is preserved.

\begin{proposition}\label{prop:not general position}
	For all $y,\bar{y}\in Y_N$ with $y\ne\bar{y}$ the following hold:
	\begin{enumerate}[label=(\roman*)]
		\item $\Phi(y)\nleq \Phi(\bar{y})$.
		\item \label{item: not general position 2} $(y_k,\dots,y_1)\llex(\bar{y}_k,\dots,\bar{y}_1)$ implies $(\Phi(y)_k,\dots,\Phi(y)_1)\llex(\Phi(\bar{y})_k,\dots,\Phi(\bar{y})_1)$.
	\end{enumerate}
\end{proposition}
\begin{proof} \phantom{linebreak}
	
	\begin{enumerate}[label=(\roman*)]
		\item  Since both $y$ and $\bar{y}$ are nondominated, there exists a $j\in[k]$ with $\bar{y}_j<y_j$. Then, 
		$$\Phi(\bar{y})_j \le \bar{y}_j+\frac{K\delta}{K+1}< \bar{y}_j+\delta < y_j\le \Phi(y)_j.$$ 
		Thus, it is $\Phi(y)\nleq \Phi(\bar{y})$.
		\item By construction, there exist $s,t\in[K]$ with $s<m$ such that the following hold for all $i\in[k]$:
		\begin{itemize}
			\item If $y_i=\bar{y}_i$, it holds that $$\Phi(y)_i\le y_i + \frac{s\delta}{K+1}<\bar{y}_i+\frac{t\delta}{K+1}=\Phi(\bar{y})
			_i.$$
			\item If $y_i<\bar{y}_i$, it holds that $$\Phi(y)_i\le y_i + \frac{s\delta}{K+1}<\bar{y}_i\le \Phi(\bar{y})_i.$$
		\end{itemize}
		Specifically,  it holds that $\Phi(y)_k<\Phi(\bar{y})_k$. Thus, it is $$(\Phi(y)_k,\dots,\Phi(y)_1)\llex(\Phi(\bar{y})_k,\dots,\Phi(\bar{y})_1).$$
	\end{enumerate}
\end{proof}
Additionally, viable combinations of $\Phi(\Y)$ and  viable combinations of $Y_N$ are strongly related. More precisely, the preimage of a viable combination of $\Phi(Y_N)$ is a viable combination of $Y_N$.
To proof this, we require the following lemma.

\begin{lemma}\label{lemma: preimage inequality}
    Let $\Phi(\Y)$ be a viable combination of $\Phi(Y_N)$. Then, for all $i,j\in[k-1]$, it holds that $\Y^j_i\le \Y^i_i$.
\end{lemma}
\begin{proof}
    In case $i=j$, the statement obviously holds. Otherwise, since $\Phi(\Y)$ is a viable combination, it holds that $\Phi(\Y^j)_i<\Phi(\Y^i)_i$ (\Cref{cor:bascis}.\ref{item:basic 1}).
    Suppose it is $\Y^i_i<\Y^j_i$. Then, by the construction of $\Phi$, it holds that
    $$\Phi(\Y^i)_i \le \Y^i_i+\frac{K\delta}{K+1} < \Y^j_i\le\Phi(\Y^j)_i$$
    which is a contradiction.
\end{proof}
\begin{theorem}\label{lemma: viable combinations}
	Let $\Phi(\Y)$ be a viable combination of $\Phi(Y_N)$. Then, $\Y$ is a viable combination of $Y_N$ .
\end{theorem}
\begin{proof}
	We show that $\e(\Y)\in \clE{\Y^1}{1}$. Showing the same for $\Y^2,\dots,\Y^{k-1}$ works analogously.  For each $m\in\mathbb{N}$ and for all $i\in[k-1]$, we define an $\e^{(m)}$  by 
	$$\e^{(m)}_i\coloneqq\begin{cases}
		\Y^i_i, & \text{if } (\Y^i_k,\dots, \Y^i_
		1)\llex (\Y^1_k,\dots,\Y^1_1) \\
		\Y^i_i+\frac{\delta}{m}, & \text{else.}  
	\end{cases}$$
	Then, it holds that $\lim_{m\to\infty} \e^{(m)}	=\e(\Y)$. 
    We show that $\e^{(m)}\in\E{\Y^1}$ for all $m\in\mathbb{N}$.
	To this end, we  first prove that $\Y^1$ is feasible for $\emethod{\e^{(m)}}$, i.\,e., $\Y^1_i<\e_i^{(m)}$ for all $i\in[k-1]$. By \Cref{lemma: preimage inequality}, it holds that $\Y^1_i \le \Y_i^i$. Thus, we need to show that $\Y^1_i<\Y^i_i=\e^{(m)}_i$ for all $i\in[k-1]$  with $(\Y^i_k,\dots, \Y^i_1)\llex (\Y^1_k,\dots,\Y^1_1)$.
 Suppose this is not the case, i.\,e., $\Y^1_i=\Y^i_i$. Then, for some $s,r$ with $s<r$ it holds that 
 $$\Phi(\Y^i)_i\le \Y^i_i+ \frac{s\delta}{K +1} < \Y^1_i +\frac{r\delta}{K+1}= \Phi(\Y^1)_i.$$ But this contradicts \Cref{cor:bascis}.\ref{item:basic 1}. Thus, $\Y^1_i<\Y^i_i$ for all $i\in[k-1]$  with $(\Y^i_k,\dots, \Y^i_1)\llex (\Y^1_k,\dots,\Y^1_1)$ and $\Y^1$ is feasible for $\emethod{\e^{(m)}}$.
 
 It remains to be shown that $\Y^1$ is also optimal for $\emethod{\e^{(m)}}$ for all $m\in\mathbb{N}$. 
Suppose this is not the case.
 Then, there exists a $y^*$ with  $(y^*_k,\dots,y^*_1)\llex (\Y^1_k,\dots,\Y^1_1)$ that is feasible for $\emethod{\e^{(m)}}$ for some $m\in\mathbb{N}$. Thus, by the definition of $\delta$, it already holds that $y^*$ is feasible for $\emethod{\e^{(m)}}$ for all $m\in\mathbb{N}.$, i.\,e., $\Y^1$ is not optimal for any.
 We show that this contradicts that $\e(\Phi(\Y))\in\clE{\Phi(\Y^1)}{1}$, i.\,e., it contradicts that $\Phi(\Y)$ is a viable combination of $\Phi(\Y_N)$. To this end, let $\bar{\e}^{(m)}$ be any sequence in $\E{\Phi(\Y^1)}$ with limit $\e(\Phi(\Y))$. We show that for all $i\in[k-1]$ and for sufficiently large $m'$, it holds that $\Phi(y^*)_i<\bar{\e}_i^{(m')}$, i.\,e., $\Phi(y^*)$ is feasible for $\emethod{\e^{(m')}}$. This then contradicts that $\bar{\e}^{(m)}$ is a sequence in $\E{\Phi(\Y^1)}$ since $(\Phi(y^*)_k,\dots,\Phi(y^*)_1)\llex(\Phi(\Y^1)_k,\dots,\Phi(\Y^1)_1)$ by \Cref{prop:not general position}.\ref{item: not general position 2}. To this end, we consider two distinct cases:
 \begin{enumerate}
     \item First, we consider all  $i\in[k-1]$ with $$(\Y^i_k,\dots, \Y^i_
		1)\llex (\Y^1_k,\dots,\Y^1_1).$$ 
        For these, we have that $y^*_i<\e^{(m)}_i=\Y^i_i$ and, thus, $\Phi(y^*)_i+\frac{\delta}{K+1}<\Phi(\Y^i)_i$.
        Since $\lim_{m\to\infty} \bar{\e}^{(m)}=\e(\Phi(\Y))$, it is $\lim_{m\to\infty} \e^{(m)}_i=\Phi(\Y^i)_i$. Consequently, for sufficiently large $m'$, it holds that $\Phi(y^*)_i<\bar{\e}_i^{(m')}.$
        
  \item Next, we consider  all $i\in[k-1]$ with $$(\Y^1_k,\dots, \Y^1_
		1)\llex (\Y^i_k,\dots,\Y^i_1).$$
        For these, it holds that $y^*_i< \e^{(m)}=\Y^i_i+\frac{\delta}{m}$ and, therefore, $y^*_i<\Y^i_i$ and 
        $$(y^*_k,\dots,y^*_1)\llex (\Y^1_k,\dots, \Y^1_
		1)\llex (\Y^i_k,\dots,\Y^i_1).$$
        Thus, it holds that $\Phi{(y^*)}_i < \Phi{(\Y^i)}_i$ and, consequently, for sufficiently large $m'$, it is $\Phi{(y^*)}_i<\bar{\e}^{(m')}_i$.
   \end{enumerate}

    Therefore, we get that, for sufficiently large $m'$, $\Phi(y^*)$ is feasible for $\emethod{\e^{(m')}}$. This contradicts that $\bar{\e}^{(m)}$ is a sequence in $\E{\Phi(\Y^1)}$.

\end{proof}
The reverse of \Cref{lemma: viable combinations} does not hold as the following example shows.
\begin{example}\label{ex:true combinations}
	Consider  a tri-objective integer optimization problem with $$
	Y=Y_N=\left\{y^1=(4,3,2)^\top,y^2=(4,2,3)^\top,y^3=(2,3,4)^\top \right\}.
	$$
	Then, for $\delta=1$, we have
	$$
	\Phi(Y_N)=\left\{\Phi(y^1)=(4,3,2)^\top,\Phi(y^2)=(4.5,2,3)^\top,\Phi(y^3)=(2,3.75,4)^\top \right\}
	$$ 
	It holds that $\Y=(y^2,y^3)$ is a viable combination of $Y_N$. But, $\clE{\Phi(y^2)}{1}\cap \clE{\Phi(y^3)}{2}=\emptyset$. Thus, $(\Phi(y^2),\Phi(y^3))$ is not a viable combination of $\Phi(Y_N$). A visualization is given in \Cref{fig:true combs}.
\end{example}

\begin{figure}[!ht]
	\centering
	\subfloat[]{
    \begin{tikzpicture}[scale = 1]
			
			\begin{axis}[
				axis lines=left,
				width=200pt,
				xmin=0,xmax=7,
				ymin=0,ymax=6,
				xtick ={5},
				ytick={5},
				xlabel={$f_1$},
				ylabel={$f_2$},	every axis x label/.style={
					at={(axis cs: 7,0,0)},
					anchor=north},
				every axis y label/.style={
					at={(axis cs: 0,6,0)},
					anchor=south
				}]
				\filldraw [fill=dummy,draw=black] (axis cs: 1,1) rectangle (axis cs:19,10);
				\filldraw [fill=i3,draw=black] (axis cs: 2,3) rectangle (axis cs: 19,10);
				\filldraw [fill=i2,draw=black] (axis cs: 4,2) rectangle (axis cs: 19,10);
				\filldraw [fill=i1,draw=black] (axis cs: 4,3) rectangle (axis cs: 19,10);
				
				\filldraw [fill=dummy,draw=black] (axis cs: 1,5) rectangle (axis cs: 19,10);
				\filldraw [fill=dummy,draw=black] (axis cs: 6,1) rectangle (axis cs: 19,10); 
				\addplot[mark=x,mark options={scale=3.2}, line width=2pt] coordinates {(6,5)} ;
				\addplot[mark=x,white,mark options={scale=3},thick] coordinates {(6,5)} ;
				\addplot[mark=x,mark options={scale=3.2}, line width=2pt] coordinates {(6,3)} ;
				\addplot[mark=x,white,mark options={scale=3},thick] coordinates {(6,3)} ;
				\addplot[mark=x,mark options={scale=3.2}, line width=2pt] coordinates {(6,2)} ;
				\addplot[mark=x,white,mark options={scale=3},thick] coordinates {(6,2)} ;
				\addplot[mark=x,mark options={scale=3.2}, line width=2pt] coordinates {(4,5)} ;
				\addplot[mark=x,white,mark options={scale=3},thick] coordinates {(4,5)} ;
				\addplot[mark=x,mark options={scale=3.2}, line width=2pt] coordinates {(2,5)} ;
				\addplot[mark=x,white,mark options={scale=3},thick] coordinates {(2,5)} ;
				\addplot[mark=x,mark options={scale=3.2}, line width=2pt] coordinates {(4,3)} ;
				\addplot[mark=x,green,mark options={scale=3},thick] coordinates {(4,3)} ;
	
			\end{axis}
			
		\end{tikzpicture}	
 
    }
	\subfloat[]{
    \begin{tikzpicture}[scale = 1]
			\begin{axis}[
				axis lines=left,
				width=200pt,
				xmin=0,xmax=7,
				ymin=0,ymax=6,
				xtick ={5},
				ytick={5},
				xlabel={$f_1$},
				ylabel={$f_2$},	every axis x label/.style={
					at={(axis cs: 7,0,0)},
					anchor=north},
				every axis y label/.style={
					at={(axis cs: 0,6,0)},
					anchor=south
				}]
				\filldraw [fill=dummy,draw=black] (axis cs: 1,1) rectangle (axis cs:19,10);
				\filldraw [fill=i3,draw=black] (axis cs: 2,3.75) rectangle (axis cs: 19,10);
				\filldraw [fill=i2,draw=black] (axis cs: 4.5,2) rectangle (axis cs: 19,10);
				\filldraw [fill=i1,draw=black] (axis cs: 4,3) rectangle (axis cs: 19,10);
				
				\filldraw [fill=dummy,draw=black] (axis cs: 1,5) rectangle (axis cs: 19,10);
				\filldraw [fill=dummy,draw=black] (axis cs: 6,1) rectangle (axis cs: 19,10);
				\addplot[mark=x,mark options={scale=3.2}, line width=2pt] coordinates {(6,5)} ;
				\addplot[mark=x,white,mark options={scale=3},thick] coordinates {(6,5)} ;
				\addplot[mark=x,mark options={scale=3.2}, line width=2pt] coordinates {(6,3)} ;
				\addplot[mark=x,white,mark options={scale=3},thick] coordinates {(6,3)} ;
				\addplot[mark=x,mark options={scale=3.2}, line width=2pt] coordinates {(6,2)} ;
				\addplot[mark=x,white,mark options={scale=3},thick] coordinates {(6,2)} ;
				\addplot[mark=x,mark options={scale=3.2}, line width=2pt] coordinates {(4,5)} ;
				\addplot[mark=x,white,mark options={scale=3},thick] coordinates {(4,5)} ;
				\addplot[mark=x,mark options={scale=3.2}, line width=2pt] coordinates {(2,5)} ;
				\addplot[mark=x,white,mark options={scale=3},thick] coordinates {(2,5)} ;
				\addplot[mark=x,mark options={scale=3.2}, line width=2pt] coordinates {(4,3.75)} ;
				\addplot[mark=x,green,mark options={scale=3},thick] coordinates {(4,3.75)} ;
				\addplot[mark=x,mark options={scale=3.2}, line width=2pt] coordinates {(4.5,3)} ;
				\addplot[mark=x,green,mark options={scale=3},thick] coordinates {(4.5,3)} ;  

			\end{axis}
			\node[rectangle,fill=i1,draw=black] (y1) at (6
			,4){};
			\node[right = of y1,xshift=-1cm,yshift=0.2cm] {$E(y^1)$};
			\node[right = of y1,xshift=-1cm,,yshift=-0.2cm] {$E(\Phi(y^1))$};
			\node[rectangle,fill=i2,draw=black, below= of y1,yshift=0] (y2) {};
			\node[right = of y2,xshift=-1cm,yshift=0.2cm] {$E(y^2)$};
			\node[right = of y2,xshift=-1cm, yshift=-0.2cm] {$E(\Phi(y^2))$};
			
			\node[rectangle,fill=i3,draw=black, below= of y2,yshift=0] (y3) {};
			\node[right = of y3,xshift=-1cm,yshift=0.2cm] {$E(y^3)$};
				\node[right = of y3,xshift=-1cm, yshift=-0.2cm] {$E(\Phi(y^3))$};

		\end{tikzpicture}
 }
	\caption{The left images depicts the epsilon-components of  the nondominated set $Y_N$ of \Cref{ex:true combinations}. All viable parameters are marked. The right image depicts the same but for the  set $\Phi(Y_N)$. The parameter that is marked in green in the left image is defined by three different viable combinations, $(y^1,y^3),(y^2,y^1)$ and $(y^2,y^3)$. Only two of these are also viable combinations of $\Phi(Y_N)$, $(\Phi(y^1),\Phi(y^3))$ and $(\Phi(y^2), \Phi(y^1))$. $(\Phi(y^1),\Phi(y^3))$ and $(\Phi(y^2), \Phi(y^1))$ define different viable parameters, both are marked in green in the right image.}
	\label{fig:true combs}
\end{figure}
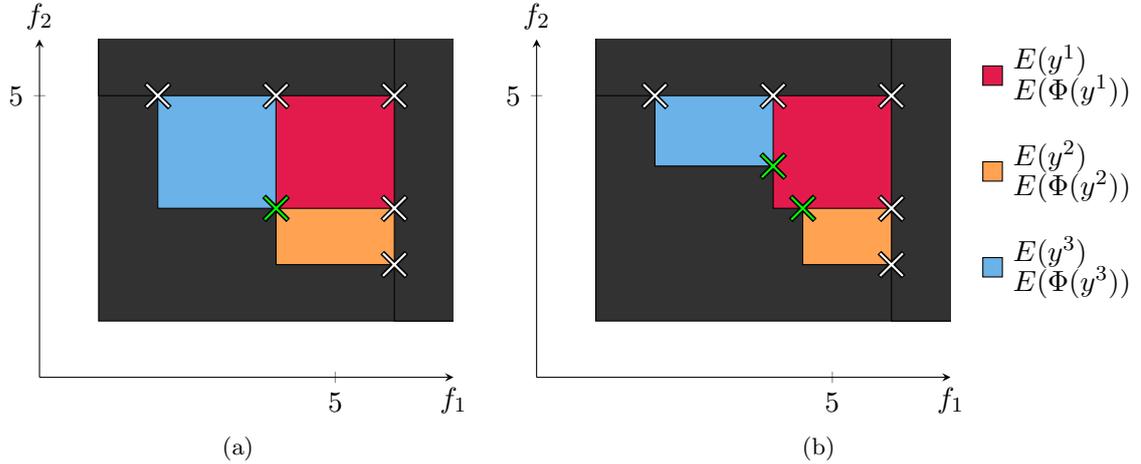
Since $\Phi(Y_N)$ is in general position, it is sufficient to solve a scalarization problem for each viable combination of $\Phi(Y_N)$ to compute it. Thus, as each nondominated imagine in $Y_N$ has a corresponding nondominated image in $\Phi(Y_N)$, it stands to reason that it is also  sufficient to solve a scalarization problem for all viable combinations of $Y_N$ that map to viable combinations of $\Phi(Y_N)$. We call these \emph{true combinations}.
\begin{definition}\label{def:true combination}
	Let $\Y$ be a viable combination of $Y_N$. Then, $\Y$ is a \emph{true combination} of $Y_N$ if $\Phi(\Y)$ is a viable combination of $\Phi(Y_N)$. We denote the set of all true combinations of $Y_N$ by $\T(Y_N)$.
\end{definition}
True combinations of $Y_N$ and the corresponding viable combinations of $\Phi(Y_N)$ ``share optimal images''.
\begin{lemma}\label{lemma: true combinations opt}
	Let $\Y$ be a true combination of $Y_N$. A nondominated image $y$ is optimal for $\emethod{\e(\Y)}$ if and only if $\Phi(y)$ is optimal for $\emethod{\e(\Phi(\Y))}$.
\end{lemma}
\begin{proof}
	   Follows directly from the construction of $\Phi(Y_N)$. 
\end{proof}
Since there is a one-to-one correspondence between $\T(Y_N)$ and $\V{\Phi(Y_N)}$ and optimal images coincide, we can define an order on $\T(Y_N)$ analogous to \Cref{def:scions} and \Cref{thm:scions}. Note that there are two significant differences: We define the order for true combinations and not viable combinations. In addition, we have strict inequalities in \Cref{def:scions} while here we only have less or equals.
\begin{theorem}\label{thm: scions not in general position}
	Let $\mathcal{Y}$ be a true combination and let $y^*$ be the optimal image of $\emethod{\e(\mathcal{Y})}$. Then, for all $\ell\in[k-1]$ such that $y^*_\ell\ge \Y^i_\ell$  for all $i\in[k-1]$ with $i\ne \ell$, it holds that $(\Y^1,\dots,\Y^{\ell-1},y^*,\Y^{\ell+1},\dots,\Y^{k-1})$ is a true combination.\\
 We call $(\Y^1,\dots,\Y^{\ell-1},y^*,\Y^{\ell+1},\dots,\Y^{k-1})$ the \emph{$\ell$-th} scion of $\Y$ and denote it by $\scion{\Y}{\ell}$.
\end{theorem}
\begin{proof}
	By \Cref{lemma: true combinations opt} we have that $\Phi(y^*)$ is  optimal for the viable combination $\Phi(\Y)$. Furthermore, by construction, $y^*_\ell\ge \Y^i_\ell$  for all $i\in[k-1]$ with $i\ne \ell$ if and only if $\Phi(y^*)_\ell > \Phi(\Y^i)_\ell$ for all $i=[k-1]$ with $i\ne \ell$. Thus, we can apply \Cref{thm:scions} and get that $(\Phi(\Y^1),\dots,\Phi(\Y^{\ell-1}),\Phi(y^*),\Phi(\Y^{\ell+1}),\dots,\Phi(\Y^{k-1}))$ is a viable combination. Therefore, by \Cref{lemma: viable combinations}, it holds that  $(\Y^1,\dots,\Y^{\ell-1},y^*,\Y^{\ell+1},\dots,\Y^{k-1})$ is a true combination.
\end{proof}
\begin{lemma}\label{lem:scion relation}
	Let $\Y,\Z\in\T(Y_N)$. Then, $\Z$ is the $\ell$-th scion of $\Y$ if and only if $\Phi(\Z)$ is the $\ell$-th scion of $\Phi(\Y)$.
\end{lemma}
\begin{proof}
	Follows directly from the proof of \Cref{thm: scions not in general position}.
\end{proof}
\Cref{lem:scion relation} allows us to transfer the properties of viable combinations and scions as stated in \Cref{sec:general position} to true combinations. Therefore, the scion order has the desired properties even if the the nondominated set is not in general position.
\begin{theorem}\label{thm:results not in general position}
	\phantom{linebreak}
	\begin{enumerate}[label=(\roman*)] 
		\item\label{item:results1}For each $\Z\in\T(Y_N)$ with $\Z\ne(d^1,\dots,d^{k-1})$, it holds that $$\left|\left\{\Y\in\T(Y_N):\exists \ell\in[k-1] \text{ with }  \Z=\scion{\Y}{\ell}\right\}\right|= 1.$$
		\item\label{item:results2} For $y\in Y_N$ there exists exactly one true combination $\Y(y)\in\T(Y_N)$ such that $y$ is optimal for $\emethod{\e(\Y(y))}$ and for all $i\in[k-1]$, it holds that $\Y(y)^i_{-[i]}\le y_{-[i]}$.
		\item\label{item:results3} It holds that $|\T(Y_N)|\in \mathcal{O}(|Y_N|^{\lfloor\frac{k}{2}\rfloor})$.
	\end{enumerate}
\end{theorem}

\begin{remark}
    If the nondominated set is not in general position, there is no one-to-one correspondence between local upper bounds,  defining points and viable/true combinations. The mapping described in \Cref{thm:map to upper bounds} only maps combinations to  so-called \emph{quasi-upper bounds} w.\;r.\;t. $Y_N$ (cf.\ D{\"a}chert et al.~\cite{Daechert2017}). In addition, upper bounds can have multiple defining points and not all correspond to viable/true combinations.
\end{remark}

\section{Algorithm} \label{sec:alg}
In this section, we formally describe PEA for computing the nondominated set of multi-objective integer optimization problems. 
The idea of PEA is as follows:
First, we initialize the dummy images. Then, we start with the true combination $(d^1,\dots,d^{k-1})$, the root of the directed tree induced by the scion order as described in \Cref{thm: scions not in general position}. Then, the directed tree is explored with depth-first search.
A complete listing can be found in \cref{alg:PEA}.

\begin{algorithm}
	\caption{PEA}
	\label{alg:PEA}
	\SetKwInOut{Input}{input}\SetKwInOut{Output}{output}
	\SetKwFunction{KwFn}{exploreSubtree}
	\SetKwProg{Fn}{function}{:}{}
	\Input{Objective function $f$, feasible set $X$.}
	\Output{Nondominated set $Y_N$.}
	\Fn{\KwFn{$\Y=(\Y^1,\dots,\Y^{k-1})$}}{
		
		$Y_N \gets \emptyset$\;
		  $y^*\gets \text{solveModel}(\emethod{\e(\Y)})$\; \label{alg: line 3}
        \If{$\Y^i_{-[i]}\le y^*_{-[i]}$ for all $i\in[k-1]$ \label{alg: line 4}}{
        $Y_N\gets Y_N\cup\{y^*\}$\; \label{alg: line 5}
        } \label{alg: line 6}
        \For{$j\in[k-1]$}{\label{alg: line 7}
		\If{$y^*_j\ge \Y^i_j$ for all $i\in[k-1]\backslash\{j\}$ }{
			$Y_N\gets Y_N \cup$ \KwFn{$\Y^1,\dots,\Y^{j-1},y^*,\Y^{j+1},\dots,\Y^{k-1}$}\; \label{alg: line 8}
		} 
	     } \label{alg: line 9} 	
  \Return $Y_N$\;
	}
	\For{$i\in[k-1]$}{

		$d^i_j\gets\begin{cases} 
			\infty, & i=j \\
			-\infty, & i\ne j
		\end{cases}$, for $j\in[k]$\;
    }
	$Y_N \gets $\KwFn{$d^1,\dots,d^{k-1}$} \label{alg: lin 13}\; 
	\Return $Y_N$\;
	
\end{algorithm}

The correctness of \Cref{alg:PEA} follows directly from \Cref{sec:not general position}. By \Cref{thm:results not in general position}.\ref{item:results1}, the scion order on the set of true combinations defines a tree. \Cref{alg:PEA} first calls the function \texttt{exploreSubtree} for the root $(d^1,\dots,d^{k-1})$ of said tree in  line~\ref{alg: lin 13}, and solves the corresponding lexicographic epsilon-constraint scalarization problem in line~\ref{alg: line 3}. We assume that \texttt{solveModel} invokes a black-box solver that correctly returns the optimal image or that the scalarization is infeasible. Then, \Cref{thm: scions not in general position} guarantees that all scions are identified in lines \ref{alg: line 7}--\ref{alg: line 9} and \texttt{exploreSubtree} is called recursively for each. Thus, \Cref{alg:PEA} goes through all true combinations. Furthermore, by \Cref{thm:results not in general position}.\ref{item:results2}, each nondominated image is added to $Y_N$ in lines \ref{alg: line 4}--\ref{alg: line 6} exactly once.

\begin{theorem}
	PEA generates the entire nondominated set and solves $\mathcal{O}(|Y_N|^{\lfloor\frac{k}{2}\rfloor})$ lexicographic epsilon-constraint scalarization problems.
\end{theorem}

The true strength of PEA lies in its ability to be parallelized in a straight-forward way:
Once called, \texttt{exploreSubtree} can work independently until the whole subtree is traversed.
Hence, any call to \texttt{exploreSubtree} can be outsourced to a new thread, and this thread can then also schedule the tasks it generates in the recursion to different threads.

The speed-up observed when parallelizing  PEA is quite significant and scales nearly linear for many instances and numbers of threads, see \Cref{sec: numerical study}.

\subsection{Improvements}
It is well known that integer programming solvers typically take longer to prove infeasibility than solving a feasible problem. In addition, providing a feasible solution can speed up the resolution of problems. See also Boland et al.~\cite{Boland2016} and Tamby and Vanderpooten~\cite{Tamby2020} for a discussion on both points. Therefore, in this section, we focus on both: We show how to avoid infeasible scalarization problems unless the multi-objective problem is already infeasible and how to provide a feasible solution  for all but the very first scalarization problem. 

To this end, we show that there exists a subset of nondominated images so that for any lexicographic epsilon-constraint scalarization problem this set contains a feasible image if such an image exists. 

This set can be constructed by applying a permutation to the objectives functions and solving a related problem. This related problem uses lexicographic epsilon-constraint scalarization problem with all $k$ objectives but minimizes them in a different order and constraints on $k-2$ objectives. Since these scalarization problems can be seen as lexicographic epsilon-constraint scalarization problems for $k - 1$ objectives functions, we refer to the
construction of the previously mentioned sets as a $k - 1$ dimensional problem.

This procedure can be applied recursively until we end up with a single-objective problem with one optimal solution, which is then feasible for all scalarization problems needed to solve the two dimensional problem, etc.

The aforementioned sets are the following:
\begin{definition}\label{def:feasible sets}
Let $r\in[k-1]$. We define
    $$Y_N(k)\coloneqq Y_N \text{ and } Y_N(r)\coloneqq \left\{ y\in Y_N(r+1): \nexists \bar{y} \in Y_N \text{ with } \bar{y}_{[r]} \le y_{[r]} \right\}.$$
\end{definition}
By definition, it holds that $Y_N(1)\subseteq \dots \subseteq Y_N(k)=Y_N$.

In order to compute the sets $Y_N(r)$, we permute the objective functions, i.\,e., we apply a permutation $\sigma$ sigma on the objective functions and given a vector $\e\in\mathbb{R}^{k-1}$ solve the following \emph{permuted lexicographic epsilon-constraint scalarization problem} 
\[
\begin{array}{ll}\tag{$\Pi_{\sigma}(\e)$}
	\lexmin & (f_{\sigma(k)}(x),f_{\sigma(k-1)}(x),\dots,f_{\sigma(1)}(x))\\
	\text{s.\,t.} & f_{-\sigma(k)}(x)<\e,\\
	& x \in X. 
\end{array}
\]
For $i\in[k]$, we use $\e_{\sigma(i)}$ to access the constraint on the objective function $f_{\sigma(i)}$. All previous results and definitions still hold as the order on the objectives was arbitrary and just fixed for simpler notation in the first place. 
We adjust the notation as follows: 
We denote the epsilon-component of a nondominated image $y$ by $E_\sigma(Y_N)$ and the set of all true combinations of $Y_N$ by $\T_\sigma(Y_N)$. Furthermore, a true combination $\Y\in\T_\sigma(Y_N)$ has the form $\Y=(\Y^{\sigma(1)},\dots,\Y^{\sigma(k-1)})$ where $\Y^{\sigma(i)}_{\sigma(i)}$ defines $\e(\Y)_{\sigma(i)}$.
We illustrate the sets defined in \Cref{def:feasible sets} and how to compute them using permuted lexicographic epsilon-constraint scalarization problems in the following tri-objective example.
\begin{example} \label{ex:permutations}
    We consider $f=\text{id}$ and
	\begin{align*}
		X=Y= \Big\{ &y^1= (5,4,2)^\top, y^2=(2,6,3)^\top, y^3=(6,2,4)^\top, y^4=(3,3,5)^\top, \\
  &y^5=(2,5,5)^\top,y^6=(5,2,6)^\top\Big\}.
	\end{align*}
    Then, it is $Y_N(2)=\left\{y^4,y^5,y^6\right\}$. We observe that  for all $\e\in\mathbb{R}^2$ the lexicographic epsilon-constraint scalarization $\emethod{\e}$ is feasible, if and only if there is a $y\in Y_N(2)$ that is feasible. Hence, $Y_N(2)$ can be used to avoid infeasible scalarization. Furthermore, since we are ``ignoring'' the third objective function, calculating $Y_N(2)$ can be done by permuting the order on the objective functions and only adjusting the bound on one objective, i.\,e., solving a bi-objective problem (where we still need to take a lexicographic minimum of all objective functions). Consider the permutation 
    $\sigma=(3, 1, 2).$
    Then, for each $y\in Y_N(2)$ there exists an $\e\in\mathbb{R}^2
    $ with $\e_{\sigma(1)}=\e_3=\infty$ such that $y$ is optimal for $\Pi_\sigma(\e)$.
    Additionally, all results from previous sections hold for any order on the objective functions. Hence, for each $y$, we can find a true combination $\Y\in\T_\sigma(Y_N)$ with $\Y^{\sigma(1)}=d^{\sigma(1)}$ such that $y$ is the optimal image. However, which permutation we use is important. Specifically $f_3$ needs to have the lowest priority in the lexicographic minimization. Otherwise,  for example for $\sigma=(1, 3, 2)$ and $\Y=(d^1,d^3)$, i.\,e., unbounded objectives, we have that $y^3$ is optimal for $\Pi_\sigma(\e(\Y))$.
    Even though, $y^3_{[2]}=(6,2)\le (5,2)= y^6_{[2]}$ and, thus, $y^3\notin Y_N(2)$.
    This procedure can be repeated recursively: We can use $Y_N(1)=\{y^5\}$ to avoid infeasible scalarizations when computing $Y_N(2)$. Here, we ``ignore'' two objective functions, use the permutation $\sigma=(3,2,1)$ and do not vary any of the parameters. 
    The epsilon-components w.\,r.\,t.\ to the different permutations are depicted in \Cref{fig:permutations}.
    
\end{example}

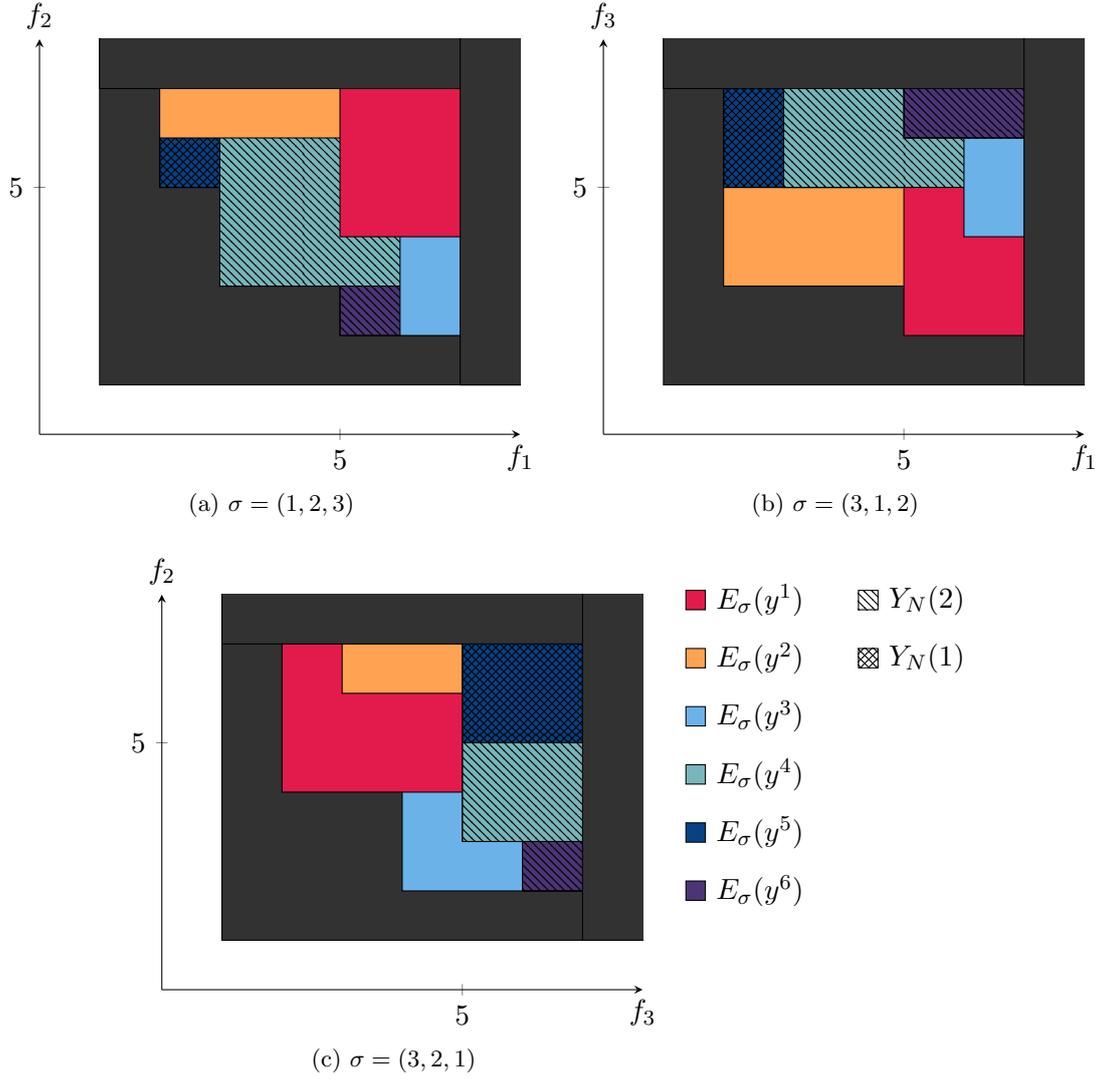
\begin{figure}[!ht]
	\centering
	\subfloat[$\sigma=(1,2,3)$]{
    \begin{tikzpicture}[scale = 1]
			
			\begin{axis}[
				axis lines=left,
				width=225pt,
				xmin=0,xmax=8,
				ymin=0,ymax=8,
				xtick ={5},
				ytick={5},
				xlabel={$f_1$},
				ylabel={$f_2$},	every axis x label/.style={
					at={(axis cs: 8,0,0)},
					anchor=north},
				every axis y label/.style={
					at={(axis cs: 0,8,0)},
					anchor=south
				}]
				\filldraw [fill=dummy,draw=black] (axis cs: 1,1) rectangle (axis cs:19,10);
    
                \filldraw [fill=i6,draw=black] (axis cs: 5,2) rectangle (axis cs: 19,10);
                \filldraw [pattern=north west lines] (axis cs: 5,2) rectangle (axis cs: 19,10);
                
                \filldraw [fill=i5,draw=black,] (axis cs: 2,5) rectangle (axis cs: 19,10);
                \filldraw [pattern=north west lines] (axis cs: 2,5) rectangle (axis cs: 19,10);
                \filldraw [pattern=north east lines] (axis cs: 2,5) rectangle (axis cs: 19,10);
                
                \filldraw [fill=i4,draw=black] (axis cs: 3,3) rectangle (axis cs: 19,10);
                \filldraw [pattern=north west lines] (axis cs: 3,3) rectangle (axis cs: 19,10);
                
                \filldraw [fill=i3,draw=black] (axis cs: 6,2) rectangle (axis cs: 19,10);              
				\filldraw [fill=i2,draw=black] (axis cs: 2,6) rectangle (axis cs: 19,10);
				\filldraw [fill=i1,draw=black] (axis cs: 5,4) rectangle (axis cs: 19,10);
	
				\filldraw [fill=dummy,draw=black] (axis cs: 1,7) rectangle (axis cs: 19,10);
				\filldraw [fill=dummy,draw=black] (axis cs: 7,1) rectangle (axis cs: 19,10);

			\end{axis}

	\end{tikzpicture}
 
    }
	\subfloat[$\sigma=(3,1,2)$]{
    \begin{tikzpicture}
        \begin{axis}[
				axis lines=left,
				width=225pt,
				xmin=0,xmax=8,
				ymin=0,ymax=8,
				xtick ={5},
				ytick={5},
				xlabel={$f_1$},
				ylabel={$f_3$},	every axis x label/.style={
					at={(axis cs: 8,0,0)},
					anchor=north},
				every axis y label/.style={
					at={(axis cs: 0,8,0)},
					anchor=south
				}]
				\filldraw [fill=dummy,draw=black] (axis cs: 1,1) rectangle (axis cs:19,10);
                \filldraw [fill=i2,draw=black] (axis cs: 2,3) rectangle (axis cs: 19,10);
                
                \filldraw [fill=i5,draw=black] (axis cs: 2,5) rectangle (axis cs: 19,10);
                \filldraw [pattern=north west lines] (axis cs: 2,5) rectangle (axis cs: 19,10);
                \filldraw [pattern=north east lines] (axis cs: 2,5) rectangle (axis cs: 19,10);
                
                \filldraw [fill=i1,draw=black] (axis cs: 5,2) rectangle (axis cs: 19,10);
                
                \filldraw [fill=i4,draw=black] (axis cs: 3,5) rectangle (axis cs: 19,10);
                \filldraw [pattern=north west lines] (axis cs: 3,5) rectangle (axis cs: 19,10);
                
                \filldraw [fill=i3,draw=black] (axis cs: 6,4) rectangle (axis cs: 19,10);
                
                \filldraw [fill=i6,draw=black] (axis cs: 5,6) rectangle (axis cs: 19,10);
                \filldraw [pattern=north west lines] (axis cs: 5,6) rectangle (axis cs: 19,10);
	               
				\filldraw [fill=dummy,draw=black] (axis cs: 1,7) rectangle (axis cs: 19,10);
				\filldraw [fill=dummy,draw=black] (axis cs: 7,1) rectangle (axis cs: 19,10);

			\end{axis}

	\end{tikzpicture}
 }\\
 \subfloat[$\sigma=(3,2,1)$]{
    \begin{tikzpicture}
        \begin{axis}[
				axis lines=left,
				width=225pt,
				xmin=0,xmax=8,
				ymin=0,ymax=8,
				xtick ={5},
				ytick={5},
				xlabel={$f_3$},
				ylabel={$f_2$},	every axis x label/.style={
					at={(axis cs: 8,0,0)},
					anchor=north},
				every axis y label/.style={
					at={(axis cs: 0,8,0)},
					anchor=south
				}]
				\filldraw [fill=dummy,draw=black] (axis cs: 1,1) rectangle (axis cs:19,10);
				\filldraw [fill=i3,draw=black] (axis cs: 4,2) rectangle (axis cs: 19,10);
				\filldraw [fill=i1,draw=black] (axis cs: 2,4) rectangle (axis cs: 19,10);
				
				\filldraw [fill=i6,draw=black] (axis cs: 6,2) rectangle (axis cs: 19,10);
				\filldraw [pattern=north west lines] (axis cs: 6,2) rectangle (axis cs: 19,10);
				
				\filldraw [fill=i4,draw=black] (axis cs: 5,3) rectangle (axis cs: 19,10);
				\filldraw [pattern=north west lines] (axis cs: 5,3) rectangle (axis cs: 19,10);
				
				\filldraw [fill=i2,draw=black] (axis cs: 3,6) rectangle (axis cs: 19,10);
				
				\filldraw [fill=i5,draw=black] (axis cs: 5,5) rectangle (axis cs: 19,10);
				\filldraw [pattern=north west lines] (axis cs: 5,5) rectangle (axis cs: 19,10);
				\filldraw [pattern=north east lines] (axis cs: 5,5) rectangle (axis cs: 19,10);
				
				\filldraw [fill=dummy,draw=black] (axis cs: 1,7) rectangle (axis cs: 19,10);
				\filldraw [fill=dummy,draw=black] (axis cs: 7,1) rectangle (axis cs: 19,10);  

			\end{axis}
        
	\end{tikzpicture}
 }
 \subfloat{
\begin{tikzpicture}
    \node[rectangle,fill=i1,draw=black] (y1) at (7,6){};
        \node[right = of y1,xshift=-1cm] {$E_\sigma(y^1)$};
        
        \node[rectangle,fill=i2,draw=black, below= of y1,yshift=0.5cm] (y2) {};
        \node[right = of y2,xshift=-1cm] {$E_\sigma(y^2)$};
        
        \node[rectangle,fill=i3,draw=black, below= of y2,yshift=0.5cm] (y3) {};
        \node[right = of y3,xshift=-1cm] {$E_\sigma(y^3)$};
        
        \node[rectangle,fill=i4,draw=black, below= of y3,yshift=0.5cm] (y4) {};
        \node[right = of y4,xshift=-1cm] {$E_\sigma(y^4)$};
        
        \node[rectangle,fill=i5,draw=black, below= of y4,yshift=0.5cm] (y5) {};
        \node[right = of y5,xshift=-1cm] {$E_\sigma(y^5)$};
        
        \node[rectangle,fill=i6,draw=black, below= of y5,yshift=0.5cm] (y6) {};
        \node[right = of y6,xshift=-1cm] {$E_\sigma(y^6)$};

        \node[draw,rectangle,pattern=north west lines, right= of y1, xshift=1cm] (Y2) {};
        \node[right = of Y2,xshift=-1cm] {$Y_N(2)$};

        \node[draw,rectangle,pattern=crosshatch,below= of Y2,yshift=0.5cm] (Y1) {};
        \node[right = of Y1,xshift=-1cm] {$Y_N(1)$};

        \node[below=of y6, yshift=-0.52cm] {};
\end{tikzpicture}
 }
	\caption{The epsilon-components of the images of \Cref{ex:permutations} for different permutations. The epsilon-components of the images in  $Y_N(2)$ and $Y_N(1)$ are highlighted by different patterns. }
	\label{fig:permutations}
\end{figure}

We now formalize the observations from \Cref{ex:permutations}. First, for all $r\in [k]$ we show how $Y_N(r)$ can be computed. To this end, we use the permutations $\sigma^r$ given by
$$
\sigma^r(j)=\begin{cases}
k-j +1, &\text{if } j\le k-r \\
r-k+j, &\text{if } j > k-r
\end{cases}, \text{ for all } j\in[k].
$$
Hence, $\sigma^r=(k,\;k-1,\; \dots, \;r+1, \; 1,\; 2,\; \dots \;r,)$. Specifically, $\sigma^k=\text{id} = (1,\; \dots,\; k)$ and $\sigma^1 = (k, \;k-1, \;\dots, \; 1)$.

We now show that if we consider permuted problems with $\sigma^r$ and leave $k-r$ objective functions unbounded, we can compute $Y_N(r)$. We do this by applying \Cref{thm:results not in general position}.
\begin{theorem}\label{thm:permutations}
    Let $y^*\in Y_N$ and $r\in[k]$. Then, $y^*\in Y_N(r)$ if and only if there exists a true combination $\Y\in\T_{\sigma^r}(Y_N)$ with $(\Y^{\sigma^r(1)},\dots, \Y^{\sigma^r(k-r)})=(d^{\sigma^r(1)},\dots,d^{\sigma^r(k-r)})$ such that $y^*$ is optimal for 
    $\Pi_{\sigma^r}(\e(\Y))$.
\end{theorem}
\begin{proof}

Note that $(\sigma^r(1),\dots,\sigma^r(k-r))=(k,k-1,\dots,r+1)$. 

First, let $y^*\in Y_N(r)\subseteq Y_N$. Thus, by \Cref{thm:results not in general position}, there exists a true combination $\Y(y^*)\in\T_{\sigma^r}(Y_N)$ such that $y^*$ is optimal for $\Pi_{\sigma^r}(\e(\Y^*))$. Additionally, for $i\in[k-1]$ it holds that $$\Y(y^*)^{\sigma^r(i)}_{\sigma^r(i+1),\dots,\sigma^r(k)} \le y^*_{\sigma^r(i+1),\dots,\sigma^r(k)}$$ (this is the equivalent of $\Y(y^*)^i_{-[i]}\le y^*_{-[i]}$ under the permutation $\sigma^r$). Specifically, by the construction of $\sigma^r$, for $i\le k-r$ we get that $Y^{\sigma^r(i)}_{[k-i]}\le y^*_{[k-i]}$. 
That means that $\Y^{\sigma^r(1)}_{[k-1]}\le y^*_{[k-1]},\dots ,\Y^{\sigma^r(k-r)}_{[r]}\le y^*_{[r]}$.
Hence, since $y^*\in Y_N(r)\subseteq Y_N(r+1)\dots \subseteq Y_N(k)$, it holds that $\Y^{\sigma^r(i)}=d^{\sigma^r(i)}$.

Conversely, let $\Y\in\T_{\sigma^r}(Y_N)$ with $$\left(\Y^{\sigma^r(1)},\dots, \Y^{\sigma^r(k-r)}\right)=\left(d^{\sigma^r(1)},\dots,d^{\sigma^r(k-r)}\right)$$ such that $y^*$ is optimal for 
    $\Pi_{\sigma^r}(\e(\Y))$ be given. Suppose  there exists a $y\in Y_N$ with $y_{[r]}\le y^*_{[r]}$. As the objectives $r+1,\dots,k$ are unbounded, this means that $y$ is also feasible for $\Pi_{\sigma^r}(\e(\Y))$.
    Furthermore, $$(y_r,\dots,y_1,y_{r+1},\dots,y_k)\llex(y^*_r,\dots,y^*_1,y^*_{r+1},\dots,y^*_k),$$
   that is,
    $$(y_{\sigma^r(k)},\dots,y_{\sigma^r(1)})\llex (y^*_{\sigma^r(k)},\dots,y^*_{\sigma^r(1)}).$$
    Hence, $y^*$ could not have been optimal for $\Pi_{\sigma^r}(\e(\Y))$ which is a contradiction. In addition, the existence of a $y$ with $y_{[i]}\le y_{[i]}^*$ for some $i>r$ would also contradict the optimality of $y^*$. Thus, $y^*\in Y_N(r)$.
\end{proof}

In the following, for $r\in[k]$, we denote by $\hat{\T}_{\sigma^r}(Y_N)$ the set of all true combinations $\Y\in\T_{\sigma^r}(Y_N)$ with $(\Y^{\sigma^r(1)},\dots, \Y^{\sigma^r(k-r)})=(d^{\sigma^r(1)},\dots,d^{\sigma^r(k-r)})$. Note, that it is $\hat{\T}_{\sigma^k}(Y_N)=\T(Y_N)$. 

We now show for every $r=2,\dots,k$ and all $\Y\in \hat{\T}_{\sigma^r}(Y_N)$ that the corresponding lexicographic epsilon-constraint scalarization $\Pi_{\sigma^r}(\e(\Y))$ is feasible if and only if a $y\in Y_N(r-1)$ is feasible. Hence, as by \Cref{thm:permutations} it is sufficient to consider true combinations in $\hat{\T}_{\sigma^r}(Y_N)$ to compute $Y_N(r)$, we can avoid infeasible scalarizations and obtain feasible solutions to speed up integer programming solvers by iteratively computing $Y_N(1),\dots,Y_N(k)$. 

\begin{theorem}
    Let $r=2,\dots,k$ and $\Y\in\hat{\T}_{\sigma^r}(Y_N)$. Then, $\Pi_{\sigma^r}(\e(\Y))$ is feasible if and only if a $y\in Y_N(r-1)$ is feasible.
\end{theorem}
\begin{proof}
Clearly, if a $y\in Y_N(r-1)$ is feasible for $\Pi_{\sigma^r}(\e(\Y))$, $\Pi_{\sigma^r}(\e(\Y))$ is feasible.

Conversely, let $\Pi_{\sigma^r}(\e(\Y))$ be feasible. Then, by \Cref{thm:permutations}, the optimal image $y^*$ is an element of $Y_N(r)$. Consequently, either $y\in Y_N(r-1)$ or there exists a $\bar{y}\in Y_N(r-1)$ with $\bar{y}_{[r-1]} \le y_{[r-1]}$. Thus, since $\e(\Y)_i=\Y^i_i=d^i_i$ for all $i>r$, $\bar{y}$ is feasible for $\Pi_{\sigma^r}(\e(\Y))$.
\end{proof}

\section{Computational Study} \label{sec: numerical study}

To investigate the running time characteristics of PEA in practice, we conduct a small exploratory computational study.

\medskip

PEA is implemented in C++17 and compiled with gcc 14.2.1. 
For solving the lexicographic epsilon-constraint scalarization problems, CPLEX 22.1.1.\ is used.
All CPLEX parameters are kept unchanged, except a lowered MIP tolerance to $10^{-6}$ and limiting the number of threads for CPLEX to one.
Furthermore, we use the oneTBB library~\cite{oneTBB} for the parallelization of PEA.
We provide the code of our PEA implementation under \href{https://gitlab.rhrk.uni-kl.de/xug28mot/pea-parallel-enumeration-algorithm}{gitlab.rhrk.uni-kl.de/xug28mot/pea-parallel-enumeration-algorithm}.
For building PEA with all necessary compiler flags, we also provide a cmake build script there.

\medskip

We compare PEA to an established parallelized algorithm and choose the algorithm AIRA from~\cite{Pettersson2020} as benchmark. 
However, issues arise in the implementation from~\cite{Pettersson2020} when using five or more objectives.
Since a reimplementation is out of scope for our paper, we only compare PEA to AIRA on problems with four objectives.
For a higher number of objectives, we only look at the running times of PEA and how these running times scale with the number of threads used.
We do not compare PEA to sequential algorithms. For tri-objective problems, Prinz and Ruzika~\cite{Prinz2024} already demonstrated that sequential algorithms cannot compete with PEA when multiple threads are available.
\medskip

To test PEA and AIRA, we use instances of the multi-objective knapsack problem (KP) and instances of multi-objective integer linear programs (ILP).
For both classes of instances, we take the instance from the sets provided by Kirlik and Sayın~\cite{Kirlik2015}.
The KP instances as well as the multi-objective assignment problem instances (AP) described in \cite{Kirlik2015} are commonly used in computational studies for multi-objective algorithms, e.\,g., \cite{Kirlik2014,Boland2016,Daechert2024}.
We do not use the AP instances in our study since the aforementioned past studies show that both KP and AP instances have similar running time behavior on a qualitative level.
The number of objectives ranges from $4$ to $10$, with varying sizes (defined by number of variables) of the instances for each number of objectives.
The number of variables ranges for the KP instances from $10$ to $100$, for the ILP instances with four objectives from $10$ to $80$, and for the ILP instances with five or more objectives from $10$ to $50$.
Increments are in steps of size $10$.
For the ILP instances, the number of constraints is exactly half the number of variables.
For each combination of problem, size and number of objectives, ten instances are used.
Newly generated instances are used whenever the original instances sets from~\cite{Kirlik2015} do not contain instances of such size or number of objective, our generation scheme is identical to that in~\cite{Kirlik2015}.
Our complete set of instances can be found under \href{https://gitlab.rhrk.uni-kl.de/xug28mot/pea-parallel-enumeration-algorithm}{gitlab.rhrk.uni-kl.de/xug28mot/pea-parallel-enumeration-algorithm}.

For each single instance, PEA and AIRA were run once for each number of threads~$\tau$ from the set $\{1,4,16,24,64,120\}$, with the exception of AIRA for $64$ and $120$ threads on the  four objective instances. 
This is because AIRA can only use at most $k!$ threads. Therefore, for four objective instances, AIRA is limited to $24$ threads. 
This is also why we use $24$ and $120$ threads, even though they are not powers of two, unlike the rest.
They allow for a direct comparison between PEA and AIRA (ignoring the aforementioned issues with AIRA for five or more objectives).

\medskip

The experiments were run on a compute server with two AMD EPYC 9554 processors, each with 64 physical cores, and 1.51 terrabyte RAM, and Gentoo Linux as operating system.
To explore a huge number of instance and number-of-threads combinations we imposed a rather strict time limit of \nicefrac{1}{2} hours.
Furthermore, we used the following logic to reduce the start of runs that would likely time out:
Let the number of threads $\tau$ that PEA/AIRA uses be fixed.
Furthermore, for a fixed instance size $n$ and number of objectives $k$, let every single run of PEA/AIRA with $\tau$ threads for instances of size $n$ and number of objectives $k$ have timed out.
In this case, we assumed that if we try to solve an instance of larger size or with more objectives with PEA/AIRA and $\tau$ threads,  it would also time out.
We directly skipped such constellations in the computational study and consider them as timed out in our analysis.
To give an example, assume that for PEA with $4$ threads all ILP instances with $10$ variables and $5$ objectives timed out.
Then any ILP instance with at least $10$ or more variables and at least $5$ or more objectives was directly seen as timed out when using $4$ threads.
However, for all runs of PEA with a different number of threads than $4$, this has no effect.
\subsection{Results}
\begin{table}
	\centering
	\resizebox{\textwidth}{!}{\begin{tabular}{llllllllll}  
			\toprule
			\multicolumn{4}{c}{Instance} & $1$ thread & $4$ threads &
			$16$ threads &$24$ threads & $64$ threads & $120$ threads
			\\
			\cmidrule(r){1-4}  \cmidrule(r){5-5} \cmidrule(r){6-6} \cmidrule(r){7-7} \cmidrule(r){8-8} \cmidrule(r){9-9} \cmidrule(r){10-10}
			k & n & $|Y_N|$ & \#P & Time (s) & Time (s) & Time (s) & Time (s) & Time (s) & Time (s)  \\
            \midrule
            4 & 10 & 11.6 & 43.4 & 0.14 & 0.09 & 0.11 & 0.12 & 0.17 & 0.29\\ 
             & 20 & 136.8 & 663.5 & 10.71 & 2.92 & 1.2 & 1.08 & 1.12 & 1.26\\ 
             & 30 & 397.6 & 2012.2 & 53.93 & 14.02 & 4.54 & 3.72 & 2.85 & 2.95\\ 
             & 40 & 1808.6 & 9978.2 & 456.21 & 115.45 & 31.44 & 22.31 & 12.06 & 10.15\\ 
             & 50 & 2881.1 & 16172.0 & (9)  & 218.2 & 57.54 & 40.71 & 20.36 & 15.56\\ 
             & 60 & 6393.8 & 36504.0 & (5)  & (9)  & 168.1 & 114.84 & 51.96 & 37.54\\ 
             & 70 & 15067.5 & 87501.0 & -  & (5)  & (9)  & 385.32 & 160.82 & 109.97\\ 
             & 80 & 25513.4 & 150422.8 & -  & -  & (8)  & (9)  & 329.68 & 218.34\\ 
             & 90 & 26235.3 & 151832.6 & -  & -  & (9)  & 690.6 & 282.99 & 191.97\\ 
             & 100 & (82983.6) & (498594.4) & -  & -  & (1)  & (1)  & (5)  & (9) \\ 
			\midrule
			5  & 10 & 16.2 & 120.6 & 0.44 & 0.17 & 0.17 & 0.18 & 0.24 & 0.34\\ 
			&20 & 161.2 & 1884.4 & 42.41 & 10.87 & 3.38 & 2.67 & 1.92 & 1.91\\ 
			&30 & 1058.7 & 16600.0 & (9)  & 173.55 & 45.01 & 30.97 & 14.60 & 10.88\\ 
			&40 & 4278.4 & 77618.5 & (3)  & (8)  & 335.01 & 225.60 & 94.31 & 63.74\\ 
			&50 & 9990.9 & 183462.4 & -  & (3)  & (9)  & (9)  & 271.45 & 179.13\\ 
			& 60 & (28222.7) & (576806.1) & -  & -  & (1)  & (2)  & (6)  & (7) \\ 
			& 70 & (28961.5) & (557116.0) & -  & -  & (1)  & (3)  & (5)  & (6) \\ 
			& 80 & (28856.0) & (518751.0) & -  & -  & -  & -  & (1)  & (1) \\ 
			& 90 & (59402.0) & (1113779.0) & -  & -  & -  & -  & -  & (1) \\ 
			\midrule
			6 & 10 & 19.7 & 256.6 & 1.44 & 0.45 & 0.29 & 0.29 & 0.35 & 0.47\\ 
			& 20 & 300.7 & 11176.8 & 324.96 & 81.41 & 21.33 & 14.85 & 7.11 & 5.50\\ 
			& 30 & 1927.2 & 92473.3 & (3)  & (8)  & 337.12 & 225.85 & 92.51 & 61.21\\ 
			& 40 & (6920.3) & (384086.4) & -  & -  & (6)  & (6)  & (9)  & (9) \\ 
			& 50 & (16192.0) & (935894.2) & -  & -  & -  & -  & (3)  & (4) \\ 
			& 60 & (15844.0) & (781350.0) & -  & -  & -  & -  & (1)  & (1) \\ 
			\midrule
			7  & 10 & 31.3 & 1230.7 & 13.99 & 3.62 & 1.18 & 0.94 & 0.72 & 0.78\\ 
			& 20 & 459.5 & 38554.7 & (7)  & 352.73 & 89.39 & 60.36 & 25.60 & 17.64\\ 
			& 30 & 4557.7 & 839874.3 & (1)  & (1)  & (2)  & (4)  & (8)  & 650.47\\ 
			& 40 & (7038.7) & (1120506.0) & -  & -  & -  & -  & (2)  & (3) \\ 
			\midrule
			8  & 10 & 29.2 & 2130.2 & 34.38 & 8.73 & 2.45 & 1.79 & 1.11 & 1.05\\ 
			& 20 & 573.6 & 223966.9 & (4)  & (7)  & (9)  & (9)  & 178.98 & 115.80\\ 
			& 30 & (3379.3) & (1188202.3) & -  & -  & -  & (1)  & (2)  & (3) \\ 
			\midrule
			9  & 10 & 39.0 & 28329.5 & (9)  & 138.68 & 34.86 & 23.41 & 9.78 & 6.65\\ 
			& 20 & (691.3) & (588866.3) & -  & (2)  & (7)  & (7)  & (9)  & (9) \\ 
			& 30 & (472.0) & (59723.0) & -  & -  & -  & (1)  & (1)  & (1) \\ 
			\midrule
			10 & 10 & 41.0 & 31658.8 & (9)  & 196.30 & 49.23 & 33.15 & 13.82 & 9.36\\ 
			& 20 & (1108.4) & (1543770.2)  & -  & -  & (1)  & (1)  & (5)  & (5) \\ 
			\bottomrule
	\end{tabular}}
	\caption{%
    Results of PEA on the KP instances.
    The instance size is denoted by $n$, and results are aggregated over all instances of a fixed combination of size and number of objectives.
    The columns $\left|Y_N\right|$ and \#P give the average number of nondominated images and scalarizations solved by PEA, respectively.
    If PEA could not solve an instance within the time limit with any number of threads, both columns are in brackets and give the averages over the remaining instances that could be solved. 
    The columns for the running time give the average running time in seconds or, if the number is in brackets, the number of instances that could be solved within the time limit.%
     \label{table:KP}}
\end{table}
\begin{table}
	\centering
	\resizebox{\textwidth}{!}{\begin{tabular}{llllllllll}  
			\toprule
			\multicolumn{4}{c}{Instance} & $1$ thread & $4$ threads &
			$16$ threads &$24$ threads & $64$ threads & $120$ threads
			\\
			\cmidrule(r){1-4}  \cmidrule(r){5-5} \cmidrule(r){6-6} \cmidrule(r){7-7} \cmidrule(r){8-8} \cmidrule(r){9-9} \cmidrule(r){10-10}
			k & n & $|Y_N|$ & \#P & Time (s) & Time (s) & Time (s) & Time (s) & Time (s) & Time (s)  \\
			\midrule
            4 & 10 & 38.4 & 159.1 & 1.25 & 0.45 & 0.34 & 0.35 & 0.41 & 0.51\\ 
             & 20 & 190.1 & 953.7 & 33.32 & 8.93 & 3.46 & 3.21 & 3.07 & 3.21\\ 
             & 30 & 451.0 & 2457.2 & 287.19 & 73.47 & 21.82 & 18.16 & 13.1 & 12.77\\ 
             & 40 & 571.6 & 3084.2 & (8)  & (8)  & 222.14 & 156.5 & 106.7 & 107.73\\ 
            \midrule
			 5 & 10 & 189.0 & 2642.9 & 27.44 & 7.09 & 2.11 & 1.59 & 1.05 & 1.04\\ 
			& 20 & 684.2 & 10968.3 & (9)  & 160.96 & 43.08 & 29.19 & 14.17 & 12.11\\ 
			& 30 & (506.8) & (7752.7) & (9)  & (9)  & (9)  & (9)  & (9)  & (9) \\ 
			& 40 & (1449.3) & (25291.3) & (4)  & (5)  & (8)  & (8)  & (9)  & (9) \\ 
			\midrule
			 6 & 10 & 134.5 & 4220.4 & 48.21 & 12.36 & 3.61 & 2.71 & 1.72 & 1.62\\ 
			& 20 & 1065.0 & 65453.9 & (6)  & (9)  & (9)  & 199.64 & 83.6 & 57.7\\ 
			& 30 & (2007.9) & (102662.5) & (3)  & (5)  & (6)  & (6)  & (8)  & (8) \\ 
			& 40 & (4085.5) & (275838.5) & (1)  & (3)  & (5)  & (5)  & (5)  & (8) \\ 
			\midrule
			7 & 10 & 402.7 & 54622.5 & (8)  & 196.85 & 49.73 & 33.47 & 13.60 & 8.84\\ 
			& 20 & (2504.9) & (577646.3) & (2)  & (2)  & (5)  & (8)  & (9)  & (9) \\ 
			& 30 & (2467.9) & (398181.7) & (1)  & (3)  & (4)  & (4)  & (6)  & (7) \\ 
			& 40 & (2387.5) & (331054.0) & -  & -  & -  & -  & (2)  & (2) \\ 
			\midrule
            8 & 10 & 544.8 & 274255.3 & (6)  & (7)  & 292.17 & 194.27 & 75.09 & 45.4\\ 
			& 20 & (1705.2) & (724967.0) & -  & (1)  & (4)  & (4)  & (5)  & (6) \\ 
			& 30 & (1578.0) & (634203.0) & -  & -  & (1)  & (1)  & (3)  & (3) \\ 
			& 40 & (1046.5) & (398243.5) & -  & -  & -  & -  & (1)  & (2) \\ 
			\midrule
			9 & 10 & 117.0 & 26342.1 & 477.89 & 119.22 & 30.66 & 20.86 & 8.97 & 6.13\\ 
			& 20 & (955.4) & (1056189.8) & -  & (1)  & (3)  & (3)  & (4)  & (5) \\ 
			\midrule
			10 & 10 & 301.5 & 504702.5 & (4)  & (7)  & (9)  & (9)  & 166.6 & 98.23\\ 
			& 20 & (107.0) & (75623.0) & -  & (1)  & (1)  & (1)  & (1)  & (1) \\  
			
			\bottomrule
	\end{tabular}}
	\caption{%
    Results of PEA on the ILP instances.
    The instance size is denoted by $n$, and results are aggregated over all instances of a fixed combination of size and number of objectives.
    The columns $\left|Y_N\right|$ and \#P give the average number of nondominated images and scalarizations solved by PEA, respectively.
    If PEA could not solve an instance within the time limit with any number of threads, both columns are in brackets and give the averages over the remaining instances that could be solved. 
    The columns for the running time give the average running time in seconds or, if the number is in brackets, the number of instances that could be solved within the time limit.%
    \label{table:ILP}
    }
	
\end{table}

To measure how the running times of PEA and AIRA scale with the number of threads, we cannot use typical measures such a \emph{speedup} or \emph{efficiency}.
Those require the running time of the algorithms when using only a single thread as baseline.
Due to the large sizes of most of the considered instances, no single threaded algorithm would terminate in a reasonable time, even with a much higher running time limit than $\nicefrac{1}{2}$ hours.
Thus, we do not have the running times of PEA/AIRA with a single thread for many instances.
Consequently, we consider the ``inverse'' of the speedup and define the \emph{slowdown} as follows. 
For a fixed instance size, the slowdown measures the average running time of PEA or AIRA with a fixed number of threads relative to the baseline of PEA with 120 threads.
More formally, let $\mathcal{I}_\mathcal{A}(n)$ be the set of all KP/ILP instances of  size $n$ that $\mathcal{A}\in \{\text{AIRA}, \text{PEA} \}$ finished within the running time limit when using using $\tau$ threads.
The slowdown then is defined by
\begin{align*} \label{equation:slowdown}
	sl_{\mathcal{A}}(n)\coloneqq\frac{1}{|\mathcal{I}_\mathcal{A}(n)|} \sum_{I\in\mathcal{I}_\mathcal{A}(n)} \frac{t_\mathcal{A}(I)}{t_{\mathrm{PEA120}}(I)},
\end{align*}
where $t_\mathcal{A}(I)$ is the running time of $\mathcal{A}$ with $\tau$ threads for instance $I$, and $t_{\mathrm{PEA120}}(I)$ is the running time of PEA with 120 threads for instance $I$.
A value of $sl_{\mathcal{A}}(n)=2$, for example, indicates that $\mathcal{A}$ with $\tau$ threads runs on average twice as long as PEA with 120 threads.

We list the average absolute running times of PEA in Table~\ref{table:KP} for KP instances and in Table~\ref{table:ILP} for ILP instances.
For four objectives, the slowdown for PEA and AIRA is visualized in Figure~\ref{fig:4slowdown}, and for PEA for five and six objectives it is visualized in Figure~\ref{fig:6slowdown}.
Note that the cardinality of the nondominated sets varies between instances, in particular for ILP instances.
Interested readers can find the exact running times of PEA and AIRA for each instance under \href{https://gitlab.rhrk.uni-kl.de/xug28mot/pea-parallel-enumeration-algorithm}{gitlab.rhrk.uni-kl.de/xug28mot/pea-parallel-enumeration-algorithm} for a more comprehensive picture. 

\medskip

We first discuss the results for the four objective instances.
For all but the smallest instances, PEA with 120 threads consistently is the fastest combination of algorithm and number of threads.
For any fixed number of threads $\tau$, PEA with $\tau$ threads is faster than AIRA with $\tau$ threads, with the only exception being the single-threaded runs on the ILP instances.
Although PEA does not scale completely proportional to the number of threads, for KP instances every increase in number of threads leads to a clear improvement in running time, and for ILP instances clear improvements show for at least 64 threads. Note that this is because the ILP instances have fewer nondominated images. Thus, less scalarization problems are solved and there is less for PEA to parallelize.
In contrast, AIRA can not provide much improvement in running time above 16 threads on both instance sets.
The results show that PEA outperforms AIRA both in the running time for fixed numbers of threads, and in scaling with the number of threads.

\medskip

For the instances with five or more objectives,  we only discuss how the running time of PEA scales with the number of threads, since we do not have AIRA as benchmark algorithm to compare against here. Overall, PEA scales just as well as it does for four objectives, though more instances time out. In addition, for a high number of objectives, the number of scalarization problems PEA solves is high  even for instances with few nondominated images. Therefore, we can observer clear improvements for up to 120 threads, even on instances with few nondominated images.


\medskip

Altogether, the results show that PEA can take advantage of high numbers of threads.
For many instances, using 120 threads provides a clear improvement in running time.
Compared to computational studies that exist in the literature (e.g.\ in~\cite{,Pettersson2020,Tamby2020,Daechert2024}), we are able to solve far larger instances in reasonable time.
To the best of our knowledge, for the largest of our instances, this is the first time that instances of this size are even considered in a computational study.

\begin{figure}[]
	\subfloat[KP\label{fig:4slowdownKP}]{
		\includegraphics[width=0.86\textwidth,trim={0cm 0.5cm 0cm 0cm},clip]{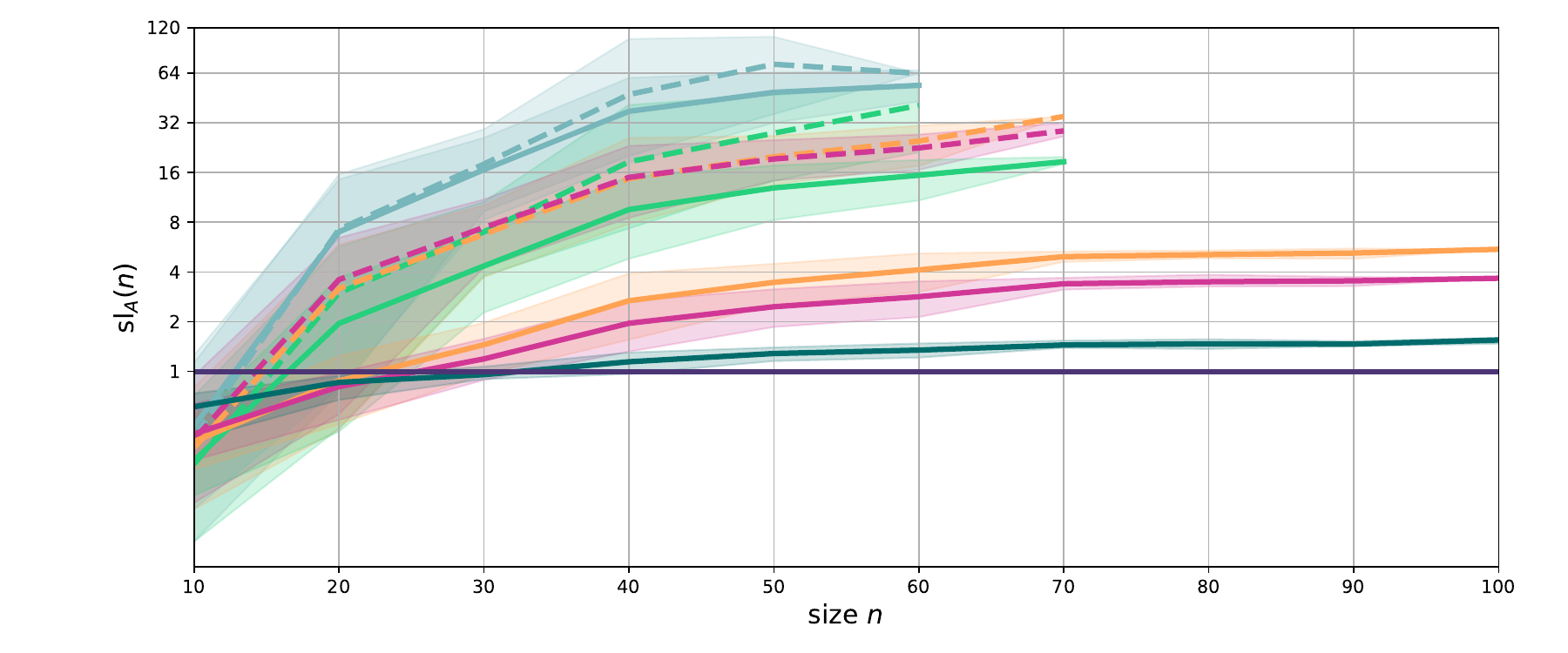}
	} 	
	
	\subfloat[ILP\label{fig:4slowdownILP}]{
		\includegraphics[width=0.86\textwidth,trim={0cm 0.5cm 0cm 0cm},clip]{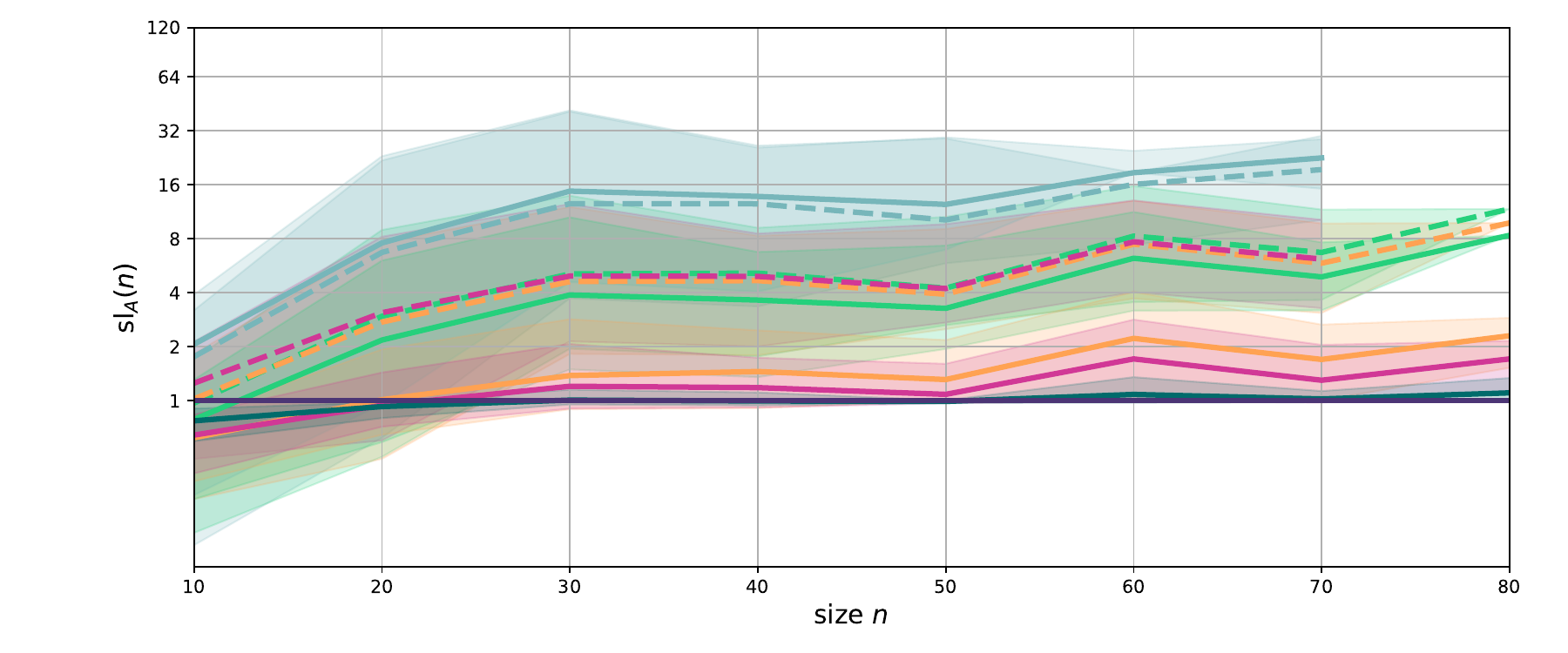}
		
	} 
	\subfloat[legend]{
		\includegraphics[width=0.1\textwidth,trim={0cm -1cm 0cm 0cm},clip]{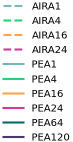}
	}
	\caption{%
    The slowdown of AIRA and PEA for KP and ILP instances with $4$ objectives.
    The naming scheme $\text{PEA}\tau$ (or $\text{AIRA}\tau$) describes the combination of algorithm and number of threads $\tau$.
    Note that the y-axes are scaled logarithmic.
    The shaded area shows the range of the slowdown over all instances  that could be solved within the time limit.
    \label{fig:4slowdown}} 
\end{figure}
\begin{figure}
	\centering
	\subfloat[KP: $k=5$]{
		\includegraphics[width=0.48\textwidth,trim={0cm 0.5cm 0cm 0cm},clip]{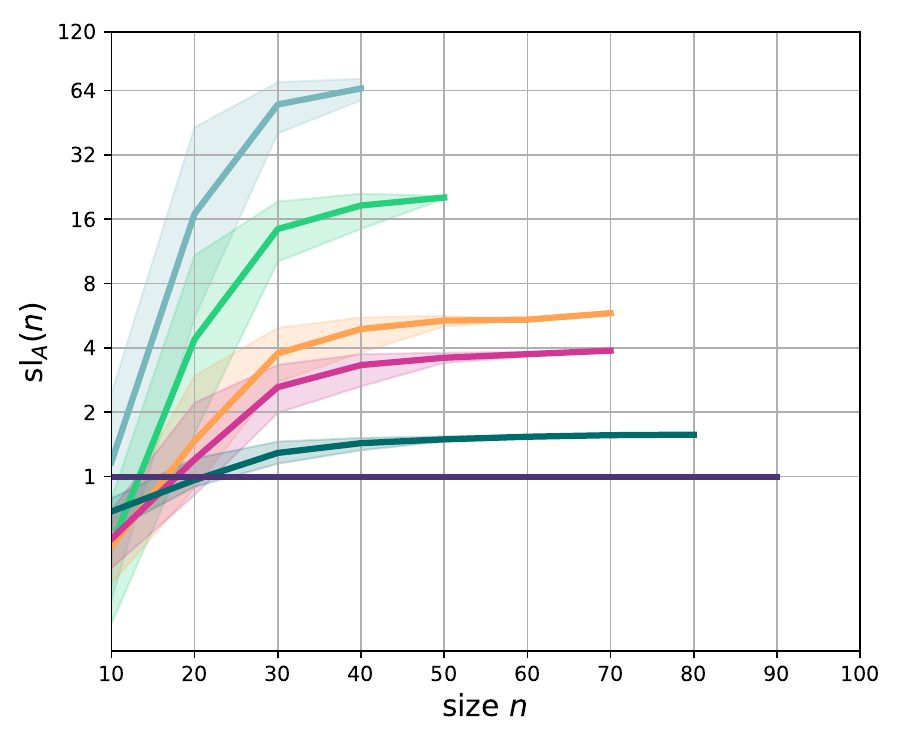}
	} 	
	\subfloat[KP:  $k=6$]{
		\includegraphics[width=0.48\textwidth,trim={0cm 0.5cm 0cm 0cm},clip]{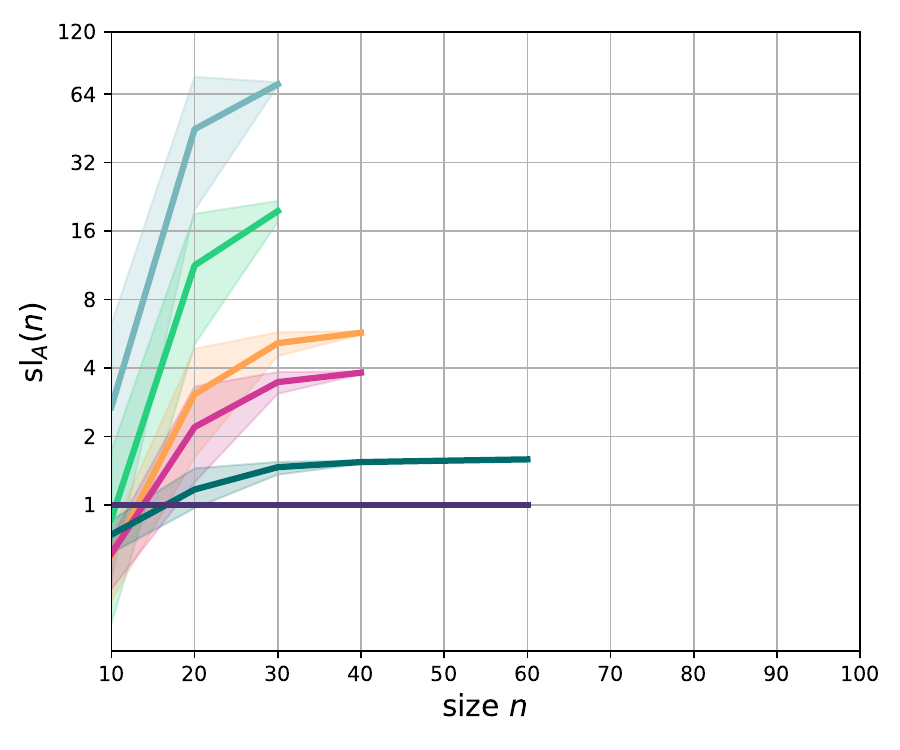}
	}
	
	\subfloat[ILP: $k=5$]{
		\includegraphics[width=0.48\textwidth,trim={0cm 0.5cm 0cm 0cm},clip]{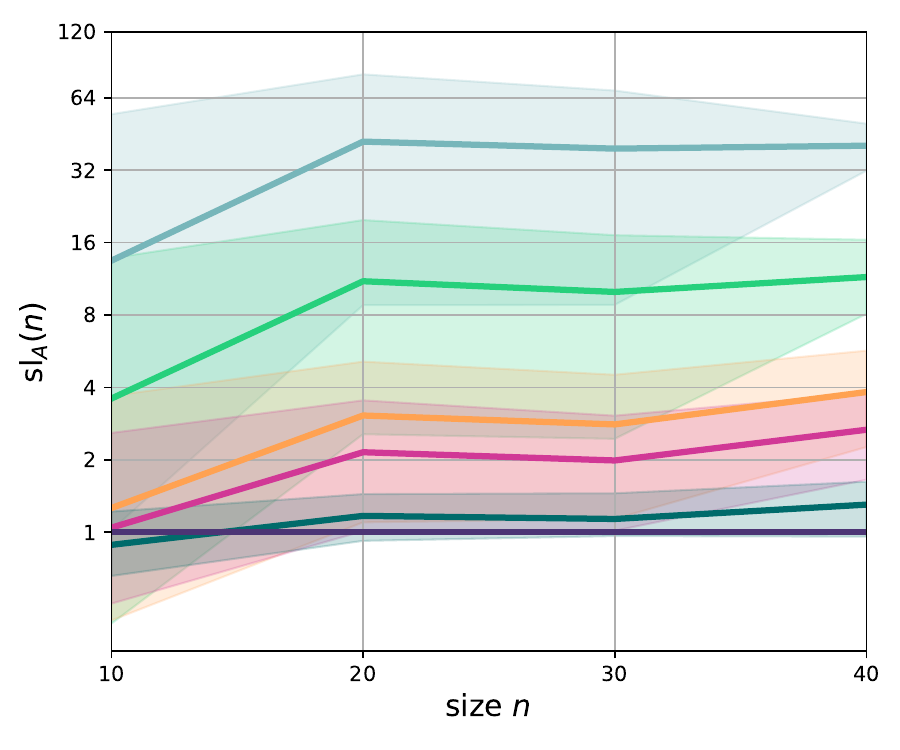}
	}  		
	\subfloat[ILP: $k=6$]{
		\includegraphics[width=0.48\textwidth,trim={0cm 0.5cm 0cm 0cm},clip]{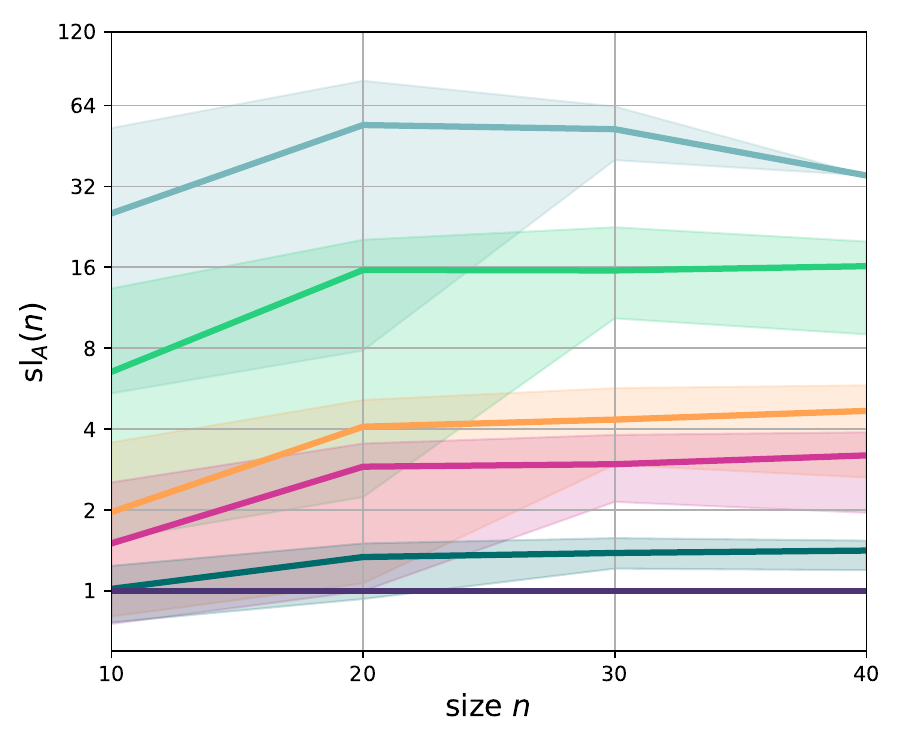}
	} 
	
	\subfloat[Legend.]{
		\includegraphics[width=0.49\textwidth]{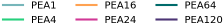}
	} 
	\caption{
    The slowdown of AIRA and PEA for KP and ILP instances with $5$ and $6$ objectives.
    The naming scheme $\text{PEA}\tau$ (or $\text{AIRA}\tau$) describes the combination of algorithm and number of threads $\tau$.
    Note that the y-axes are scaled logarithmic.
    The shaded area shows the range of the slowdown over all instances  that could be solved within the time limit.%
    \label{fig:6slowdown}} 
\end{figure}

\section{Conclusions}
\label{sec:conclusions}

We introduced a new order for parameters of epsilon-constraint scalarizations.
This order arranges the scalarizations in a directed tree.
Traversing this tree is a new algorithmic approach to find all nondominated images of a multi-objective integer problem.
With PEA, we presented the first algorithm using this approach.
The computational study  shows that it greatly speeds up the computation of nondominated images on practical instances.

Hence, the significance of PEA is two-fold:
First, it gives practitioners a new tool to utilize computational resources efficiently and to speed up many real-world applications.
Second, it proposes a new approach for the design of multi-objective optimization algorithms.
We hope that future research is able to build upon PEA to engineer faster variants or to enable it to be used for even more problem classes.
Here, we want to remark that although we described PEA for integer problems, it can be applied to any multi-objective problem with finite nondominated set, as long as a solver for the  lexicographic epsilon-scalarization problems is available.


\section*{Acknowledgments}
The authors gratefully acknowledge the funding by the Deutsche Forschungsgemeinschaft (DFG, German Research Foundation) --- GRK 2982, 516090167 ``Mathematics of Interdisciplinary Multiobjective Optimization'',  the Carl Zeiss Foundation --- Project number P2019-01-005, and the Deutsche Forschungsgemeinschaft (DFG, German Research Foundation) --- Project number 508981269.

\smallskip

\noindent%
In addition, we would like to thank William Pettersson for his technical support with his implementation of AIRA.
\bibliography{references}
\end{document}